\documentclass[a4paper, 10pt, twoside, notitlepage]{amsart}

\usepackage[utf8]{inputenc}
\usepackage{color}
\usepackage{amsmath} 
\usepackage{amssymb} 
\usepackage{amsthm}
\usepackage{geometry}
\usepackage{graphicx}
\usepackage{esint}
\usepackage[colorlinks=true,linkcolor=blue]{hyperref}
\usepackage{tikz}

\usepackage[scr=pxtx]{mathalfa}
\usepackage{enumitem}

\theoremstyle{plain}
\newtheorem{thm}{Theorem}[section]
\newtheorem{prop}[thm]{Proposition}
\newtheorem{lem}[thm]{Lemma}
\newtheorem{cor}[thm]{Corollary}

\newtheorem{rmk}[thm]{Remark}

\newcommand {\R} {\mathbb{R}} \newcommand {\Z} {\mathbb{Z}}
 \newcommand {\N} {\mathbb{N}}
\newcommand {\C} {\mathbb{C}}

\newcommand {\p} {\partial}

\newcommand {\D} {\Delta}

\newcommand {\supp} {\text{supp}}

\newcommand{\tx}{\tilde{x}}
\newcommand{\ty}{\tilde{y}}

\DeclareMathOperator {\dist} {dist}
\DeclareMathOperator {\re} {Re}
\DeclareMathOperator {\im} {Im}

\DeclareMathOperator {\sign} {sgn}

\DeclareMathOperator{\conv} {conv}
\DeclareMathOperator{\F} {\mathcal{F}}

\DeclareMathOperator{\csch}{csch}

\title[Quantitative Unique Continuation for Nonlocal Operators]{On Two Methods for Quantitative Unique Continuation Results for Some Nonlocal Operators}

\author{Mar\'ia \'Angeles Garc\'ia-Ferrero}
\address{Max-Planck Institute for Mathematics in the Sciences, Inselstraße 22, 04103 Leipzig, Germany}
\email{Maria.Garcia@mis.mpg.de}

\author{Angkana Rüland}
\address{Max-Planck Institute for Mathematics in the Sciences, Inselstraße 22, 04103 Leipzig, Germany}
\email{rueland@mis.mpg.de}

\begin{document}
\maketitle
\begin{abstract}
In this article we present two mechanisms for deducing logarithmic quantitative unique continuation bounds for certain classes of integral operators. In our first method, expanding the corresponding integral kernels, we exploit the logarithmic stability of the moment problem. In our second method we rely on the presence of branch-cut singularities for certain Fourier multipliers. As an application we present quantitative Runge approximation results for the operator $ L_s(D) = \sum\limits_{j=1}^{n}(-\p_{x_j}^2)^{s} + q$ with $s\in [\frac{1}{2},1)$ and $q\in L^{\infty}$ acting on functions on ~$\R^n$.
\end{abstract}

\tableofcontents
\addtocontents{toc}{\setcounter{tocdepth}{1}}

\section{Introduction}

In this article we present two mechanisms of obtaining logarithmic stability estimates for a class of nonlocal operators which are of significance in quantitative estimates in control theory, inverse problems and (Runge type) approximation results.
Let us describe the set-up of these problems: Often in inverse problems or control theory, one is interested in studying injectivity and stability properties associated with the \emph{inverse} of a \emph{compact} operator $T: X \rightarrow Y$ mapping between two infinite-dimensional function spaces, e.g. between two infinite-dimensional Hilbert spaces. A prototypical class of operators to have in mind are integral operators
\begin{align}
\label{eq:T}
T f(x) = \int\limits_{\R^n} k(x,y) f(y) dy
\end{align}
with smooth integral kernels $k(x,y)$. Here, for instance, we may first consider the operator acting on $f \in C_c^{\infty}(I)$ and then extend it to some suitable Hilbert space. In the following we will always assume that $I \subset \R^n$ is open and bounded and that the operators $T$ are sampled/ measured on a domain $J \subset \R^n$ which is also bounded and open. If $\overline{I}\cap \overline{J} = \emptyset$ and if $k(x,y)$ is sufficiently regular for $y\in I$ and $x\in J$, the map $f\mapsto Tf$ is compact. Typical applications involving operators of this type include the inversion properties of the Hilbert transform which arise in medical imaging \cite{ADK15, APS14, Rue17}, the fractional Calder\'on problem \cite{GSU16, RS17, RS18a} or quantitative control theoretic problems as in \cite{FZ00, LR95, RS17b, RS17a}. Assuming $T$ to act on infinite dimensional spaces, the compactness of $T$ and the non-compactness of the unit ball in infinite dimensions implies that the inverse of $T$ \emph{cannot} be continuous. Yet, general functional analytical results \cite{T66, B89} entail that after restricting to a certain compact subset $K \subset X$, the operator $T^{-1}$ becomes continuous with some modulus of continuity $\omega(t)$, i.e.
\begin{align*}
\|f\|_{X} \leq \omega(\|T f\|_{Y}), \ \|f\|_{K} \leq 1,
\end{align*}  
where $\omega(t)$ is a monotone function satisfying $\omega(0)=0$. It is then of interest to provide bounds on the modulus of continuity $\omega$ depending on the choices of the spaces $X$, $Y$ and $K$. In this article, we present two mechanisms of deducing such bounds with logarithmic moduli, i.e. for a certain class of integral operators of the form \eqref{eq:T}, we prove estimates of the type
\begin{align}
\label{eq:log_est}
\|f\|_{L^2(I)}
& \leq C \frac{\|f\|_{H^1(I)}}{\left| \log\left( C \frac{\|f\|_{H^1(I)}}{\|T f\|_{L^2(J)}} \right) \right|^{\nu}},
\end{align}
where in most cases the constants $\nu, C>0$ only depend on the (relative) geometry of $I, J$ and the dimension. Using the methods from \cite{KRS20}, it can be seen that in our applications these estimates are optimal (up to the choice of the constant $\nu$). 

For non-local operators it is in general rather hard to deduce quantitative stability/unique continuation bounds of the form \eqref{eq:log_est}, as for general nonlocal operators the most commonly used tools in proving quantitative unique continuation -- Carleman estimates and frequency function bounds -- are often not directly available (see however the results in \cite{FF14,FF15,Rue15,GFR19,S15, RS17, Rue17, RS17a, GFR19, BG18, LLR19} for qualitative and quantitative bounds for the fractional Laplacian and related operators which can be addressed by adapting local techniques by virtue of the Caffarelli-Silvestre extension).
The methods which we present here thus offer alternatives to these arguments for certain classes of nonlocal operators and are based on
\begin{itemize}
\item[(i)] the stability estimates of the \emph{(Hausdorff) moment problem},
\item[(ii)] the existence of \emph{branch-cut singularities} in symbols of certain pseudodifferential operators.
\end{itemize}
For both arguments we heavily exploit quantitative analytic continuation results, see for instance \cite{Vessella99}.
Settings to which these methods apply include the stability properties of the Hilbert, the Fourier, the Laplace, the Fourier-Laplace transforms and the one-dimensional fractional Laplacian (for the method (i)), and pseudodifferential operators of the form
\begin{align*}
P(D) = |D_{x_n}|^{2s} + L(D) + m(D'), \ s \in \R \setminus \Z,
\end{align*}
where $L(D)= \sum\limits_{\alpha \in (\N\cup \{0\})^n, \ |\alpha|\leq m} a_{\alpha} D^{\alpha}$ is a constant coefficient local operator and $m(D')$ is a pseudodifferential operator with constant coefficients depending on the derivatives $\p_{x_1}, \dots, \p_{x_{n-1}}$ (for the method (ii)). As an application of the method from (ii) we deduce Runge approximation results for the operator $ L_s(D)= \sum\limits_{j=1}^{n}(-\p_{x_j}^2)^{s}+q$ with $s\in [\frac{1}{2},1)$ and $q\in L^\infty$:

\begin{thm}
\label{thm:quant_Runge}
Assume that $\Omega = (-1,1)^n \subset \R^n$ and that $W \subset \Omega_e$ is an open, bounded domain. 
Let $q\in L^{\infty}(\Omega)$, let $L_s(D):= \sum\limits_{j=1}^{n}(-\p_{x_j}^2)^s +q$ with $s\in [\frac{1}{2},1)$ and suppose that $q$ is such that zero is not a Dirichlet eigenvalue of the operator $L_s$. Assume further that $v\in H^{s}_{\overline{\Omega}}$ and that $\epsilon>0$. Let $P_q$ denote the Poisson operator of $L_{s}$. Then there exist a constant $\mu>0$ and functions $f_{\epsilon} \in C_c^{\infty}(W)$ such that
\begin{align*}
\|P_q f_{\epsilon} - v \| \leq \epsilon \|v\|_{H^{s}_{\overline{\Omega}}}, \ \ \|f_{\epsilon}\|_{H^{s}(W)} \leq C e^{C \epsilon^{-\mu}}\|v\|_{L^2(\Omega)}.
\end{align*} 
\end{thm}

This provides another method of quantifying the qualitative bounds which had been derived in, for instance, \cite{DSV14, DSV16, GSU16}. Previously quantifications of these had always relied on stability results for the Caffarelli-Silvestre extension \cite{RS17, RS17a}. Hence only special classes of nonlocal operators could be treated with these methods.

While the quantitative approximation result of Theorem \ref{thm:quant_Runge} is new, we emphasize that it is not our primary objective in this article to discuss new stability estimates for individual operators, but to present a \emph{unified} and relatively \emph{robust} framework for a number of interesting (constant coefficient, nonlocal) operators. In particular, partially, the stability properties of the operators which are discussed in the present article have been treated in the literature already -- yet often with methods which are problem specific and do not generalize easily to larger classes of nonlocal operators.  

\subsection{Stability by exploiting the stability of the moment problem}

The Hausdorff moment problem deals with the inverse problem of recovering a function $f$ in a suitable function space from its moments $\mu(f):=\{\int_{\R^n} f(x) x^{j} dx: \ j \in (\N\cup \{0\})^{n}\}=\{f_j\}_{j\in(\N\cup\{0\})^n}$. This problem is severely ill-posed in the sense that in general it is not solvable and, if it is solvable, it is only stable with a logarithmic modulus of continuity in suitable Sobolev spaces, see  \cite{AGVT04,Talenti87} and the discussion below. In the sequel, we show that the stability properties of certain integral operators can be reduced to a combination of this problem and quantitative unique continuation estimates by a Taylor expansion of the integral kernel. 

A typical sample result of this is the logarithmic stability estimate for the Hilbert transform $H: L^2(\R) \rightarrow L^2(\R)$, $f\mapsto Hf(x):=\frac{1}{\pi} \text{p.v.} \int\limits_{\R } \frac{f(y)}{x-y}dy$, which has the Fourier symbol $- i \sign(\xi)$. For this specific problem, the logarithmic stability estimates had previously been been derived in the literature with other methods, see \cite{APS14} and the subsequent work \cite{Rue17} as well as \cite{ADK16, ADK15}. For a further discussion we refer to Section \ref{sec:Hilbert} below.

\begin{thm}
\label{thm:qucHilbert_a}
Let $I$ and $J$ be two  open, bounded, non-empty subsets of $\R$ such that $\overline I\cap\overline J=\emptyset$. Then there exists $C=C(I,J)>0$ such that for any $f\in C^\infty_c(I)$
\begin{align*}
\|f\|_{L^2(I)}
&\leq C e^{C\frac{\| f \|_{H^1(I)}}{\| f\|_{L^2(I)}}}\|Hf\|_{L^2(J)}.
\end{align*}
\end{thm}

\begin{rmk}\label{rmk:Jmeasure}
In contrast to the results on the stability for the Hilbert transform which had previously been discussed in the literature, our result follows also if $J$ is only a measurable, bounded subset with positive measure and such that $\overline I\cap\overline J=\emptyset$. This observation holds true also for the other results which are presented in Part 1 of this article. This relies on the unique continuation results from \cite{Vessella99, AE13} which are presented in Lemma \ref{lem:analytcont} below.
\end{rmk}

\begin{rmk}
\label{rmk:relax}
As can be observed from the proof of Theorem \ref{thm:qucHilbert_a}, the results of Theorem \ref{thm:qucHilbert_a} (and also of the other results in Part \ref{part_1}) can be formulated for functions $f\in H^{1}(I)$ by introducing a sharp cut-off (which is also the original formulation of the results in \cite{APS14} and \cite{Rue17}). In the setting of Theorem \ref{thm:qucHilbert_a} this leads to the estimate
\begin{align}
\label{eq:truncatedHT}
\|f\|_{L^2(I)} \leq C e^{C \frac{\|f\|_{H^1(I)}}{\|f\|_{L^2(I)}}} \|H (f\chi_{I})\|_{L^2(J)} \mbox{ for } f\in H^1(I),
\end{align}
where $\chi_I$ denotes the indicator function of the interval $I$. Hence, \eqref{eq:truncatedHT} provides a quantitative stability result for the truncated Hilbert transform $H_{I,J}:=\chi_J H \chi_I$ (analogously truncated versions can be obtained for all other operators in Part \ref{part_1}).
\end{rmk}

Variations of Theorem \ref{thm:qucHilbert_a} include logarithmic stability estimates for perturbations of the Hilbert transform, such as the modified Hilbert transform, which arises in the study of water waves and for which no Carleman estimates and no quantitative unique continuation results are known (see Section ~\ref{sec:mod_Hilb}).
While this method is particularly powerful in one dimension, we emphasise that it can also be extended to higher dimensional settings. For instance, it allows us to deal with the Fourier-Laplace transform (see Section ~\ref{sec:moment_gen}). We refer to the specific sections for the precise statements and references to earlier works on these operators.

\subsection{Stability by exploiting the presence of branch-cut singularities in the symbols of the operators}

Our second method relies on a quantitative argument which in its  qualitative form had already been observed in different variations in  \cite[Lemma 3.5.4]{Isakov}, in \cite[Section ~5]{RS17a} and from a more microlocal perspective in \cite{L82}. These results are closely related to the antilocality of the operator under consideration, which plays an important role in the physics literature in the form of Reeh-Schlieder type theorems (see for instance \cite{V93} and the references therein). Following the book of Isakov, one version of these observations can be cast into the following rather general framework:

\begin{thm}[Lemma 3.5.4 in \cite{Isakov}]
\label{thm:Isakov}
Let $\mu_j$ with $j\in\{1,2\}$ be measures with $\supp(\mu_j) \subset B_r \subset \R^n$. Let $E \in \mathcal{S}'(\R^n)$. Assume that $\F E$ cannot be written as the sum of a meromorphic function (in $\C^n$) and a distribution supported on the zero set of some non-trivial entire function. Then, if $E\ast (\mu_1 -\mu_2) = 0$ in $\R^n \setminus B_r$, we have $\mu_1 = \mu_2 $ globally.
\end{thm}

In \cite{RS17a} we showed that, due to the presence of the branch-cut singularity in {the analytic continuation of the symbol $|\xi|^{2s}$ for $s\in \R \setminus \Z$}, Theorem ~\ref{thm:Isakov} applies, for instance, to the operator $(-\D)^s$. Building on a (qualitative) idea from \cite{L82} we here provide \emph{quantitative} variants of these type of results. In this context a sample result is the following:

\begin{thm}
\label{thm:s_gen_1D}
Let $s \in[ \frac{1}{2},1)$.
Let $I, J_1, J_2 \subset \R$ be  open, connected, bounded, non-empty intervals with $\overline{I}\cap \overline{J_j} = \emptyset$ and $J_1$ located to the left of $I$ and $J_2$ to the right of $I$. Let $J=J_1\cup J_2$. Then, there exist constants  $C>0$, $\mu>0$ such that for any $g \in C_c^{\infty}(I)$ we have
\begin{align*}
\|g\|_{H^{2s}(I)} \leq C e^{C \left(\frac{\|g\|_{H^{2s}(I)}}{\|g\|_{L^2(I)}} \right)^{\mu}} \||D|^{2s} g\|_{H^{-s}(J)}.
\end{align*}
\end{thm}

This result had previously been already deduced in \cite{RS17} in $n$-dimensions relying on Carleman estimates for the Caffarelli-Silvestre extension. The main novelty here is the substantially simplified method of proof of this result which relies on analytic propagation and ``comparison-type arguments''. While not being as robust as Carleman estimates -- which, for instance, allow for extensions to variable coefficient settings -- for constant coefficient operators the method seems to allow for substantially simpler arguments.

While in this article, by the described second method, we mainly provide quantitative results for the fractional Laplacian in the one-dimensional setting, these can also be extended to yield new quantitative results on operators in higher dimensions, in particular on (not necessarily elliptic) combinations of operators involving local and pseudodifferential ingredients:

\begin{thm}
\label{thm:s_gen_nD}
Let $I, J_1,J_2, J \subset \R$ be as in Theorem ~\ref{thm:s_gen_1D} and let $Q\subset\R^{n-1}$ be an open, bounded set.
Consider the subsets
$\mathcal I= Q\times I \subset \R^n$ and $\mathcal J=Q\times J \subset \R^n$ and let
\begin{align}
P(D)=|D_{x_n}|^{2s}+L(D)+m(D')
\end{align}
where $s\in [\frac{1}{2},1)$, $L(D)$ is a constant coefficient local operator  and $m(D')$ is a pseudodifferential operator in the derivatives $\partial_{x_1},\dots,\partial_{x_{n-1}}$ such that $m(D')g|_{\mathcal J}= 0$ for all functions which are compactly supported in $\mathcal I$. Then, there exist $C>0, \mu>0$ such that for  any $g \in C_c^{\infty}(\mathcal I)$ we have
\begin{align*}
\|g\|_{H^{2s}(\mathcal{I})} \leq C e^{C \left(\frac{\|g\|_{H^{2s}(\mathcal I)}}{\|g\|_{L^2(\mathcal I)}} \right)^{\mu}} \|P(D) g\|_{H^{-s}(\mathcal J)}.
\end{align*}
\end{thm}

Similarly as in the qualitative results in \cite{DSV16,RS17a} we do not require any ellipticity/parabolicity on these operators. In future work we seek to extend our methods to quantitative settings involving the $n$-dimensional fractional Laplacian and related operators.

\subsection{Organisation of the remaining article} The remaining article is organised into two main parts: On the one hand, Part \ref{part_1} which includes Sections ~\ref{sec:moment_aux}-\ref{sec:moment_gen} deals with the ideas and results obtained through the stability of the moment problem. Here, for instance, we provide the proof of Theorem ~\ref{thm:qucHilbert_a}.
On the other hand, in Part \ref{part_2}, which consists of Sections ~\ref{sec:aux}-\ref{sec:applic}, we present the ideas, results and applications of the method exploiting the branch-cut singularities of the underlying symbols. As examples of this strategy we explain the arguments for Theorems ~\ref{thm:quant_Runge}, ~\ref{thm:s_gen_1D} and ~\ref{thm:s_gen_nD}.

\subsection{Notational conventions}
\label{subsec:conv}
In the sequel, we will use the following conventions: We denote the Fourier transform by
\begin{align*}
\F f(\xi) = \int\limits_{\R^n} f(x) e^{-ix\cdot \xi} dx
\end{align*}
for $f\in \mathcal{S}$.
We  also use the notation $\hat f=\mathcal F f$. We further set $D_{x_j}:= i \frac{\p}{\p x_j}$; in one dimension we also omit the subindex and simply write $D$.

For $s\in \R$, the whole space Sobolev spaces are denoted by 
\begin{align*}
H^s(\R^n):=\{f\in \mathcal{S'}(\R^n): \|(1+|\cdot|^2)^{\frac s 2}\hat f \|_{L^2(\R^n)} < \infty\},
\end{align*}
and the homogeneous version by
\begin{align*}
\dot H^s(\R^n):=\{f\in \mathcal{S'}(\R^n): \||\cdot|^s \hat f \|_{L^2(\R^n)} < \infty\}.
\end{align*}
Let $\Omega\subset \R^n$ be an open set, then we define
\begin{align*}
H^s(\Omega)&:=\{f|_\Omega: f\in H^s(\R^n)\},\\
\dot H^s(\Omega)&:=\{f|_\Omega: f\in \dot H^s(\R^n)\},\\
\tilde H^s(\Omega)&:=\mbox{ closure of } C^\infty_c(\Omega) \mbox{ in } H^s(\R^n),\\
 H^s_{\overline{\Omega}}&:=\{f\in H^s(\R^n): \supp f\subset \overline{\Omega}\}.
\end{align*}
For any $s\in\R$ we have
\begin{align*}
(H^s(\Omega))^*=\tilde H^{-s}(\Omega), \ \ (\tilde H^s(\Omega))^*= H^{-s}(\Omega).
\end{align*}
If in addition $\Omega$ is a bounded Lipschitz domain, the following identifications hold: 
\begin{align*}
H^s_{\overline\Omega}&=\tilde H^s(\Omega), \ s\in\R.
\end{align*}

In order to denote that a quantity $a$ is chosen of the order of another quantity $b$, we use the notation $a \simeq b$ by which we mean that there exist a constant $C>1$ (which does not depend on any relevant parameters) such that 
\begin{align*}
C^{-1} a \leq b \leq C a.
\end{align*}

We denote by $\N_0$ the set of natural numbers including zero, i.e. $\N_0:=\N\cup\{0\}$.

Given a bounded subset $I\subset\R^n$, we define $I^{-1}:=\{x\in\R^n: x^{-1}\in I\}$. Moreover, if $I=\{x\in\R:|x-x_0|\leq \ell\}$, for any $k>0$  we define $kI:=\{x\in\R:|x-x_0|\leq k\ell\}$.
If $I\subset \R$, we also use the notation $\mathcal I=Q\times I\subset \R^n$, where $Q\subset \R^{n-1}$ is a bounded subset. Then, for any $k>0$, we set $k\mathcal I:= Q\times kI$.
We denote by $\conv(J)$ the convex hull of  any subset $J\subset \R^n$, i.e.  the smallest convex set that contains $J$. By writing $I_1 \Subset I_2$ for $I_1, I_2 \subset \R^n$ open and $I_1$ bounded, we mean that $I_1$ is compactly contained in $I_2$, i.e. $\overline{I}_1 \subset I_2$, where $\overline{I_1}$ denotes the closure of $I_1$.

We use the notation $\R_+=[0,\infty)$ and $\R_-=(-\infty, 0]$.
Moreover, $\R^n_\pm=\R^{n-1}\times \R_\pm$.
Finally, we use the following notation for the balls in $\C^n$: for any $x_0\in\C^n$ and $r>0$, let $B_r(x_0):=\{x\in\C^n: |x-x_0|<r\}$.
Moreover, let $\C_\pm=\{z\in\C: \pm\im z\geq0\}$.
If $x_0\in\R^n$, we define  $B^\pm_r(x_0):=B_r(x_0)\cap (\C_\pm)^n=\{x\in (\C_\pm)^n: |x-x_0|\leq r\} $. If $x_0=0$, the dependence on $x_0$ is omitted.

\part{Stability Estimates by Exploiting the Hausdorff Moment Problem}
\label{part_1}

In this part we discuss our first strategy of deducing quantitative stability estimates for a class of nonlocal operators by reducing these to stability estimates for the moment problem. As examples of this method we present one-dimensional examples such as the Hilbert transform or the modified Hilbert transform, but also show that this is applicable to multi-dimensional problems such as the higher dimensional Fourier and Fourier-Laplace transforms.

\section{Auxiliary Results on the Moment Problem}
\label{sec:moment_aux}

\subsection{On moments}
In this section we recall the stability of the moment problem and introduce corresponding dual estimates. Later,  we  apply these results to bound the $L^2$-norm of a function by the sum of monomials weighted with the moments.

The  moment of order $j\in\N_0$ of a real-valued function $f$ supported on $[0,1]$ is the integral of the product of $f$ with the monomial $x^j$: 
\begin{align*}
f_j&=\int_{0}^1 x^j f(x) dx.
\end{align*}
In higher dimensions,  given   $f$ supported in $[0,1]^n$, the (generalized) moment of order  $j=(j_1,\dots,j_n)\in\N_0^n$ is the integral of the product of $f$ with  $x^j=x_1^{j_1}\dots x_n^{j_n}$: 
\begin{align}\label{eq:defmom}
f_j&=\int_{(0,1)^n} x^j f(x) dx.
\end{align}

The (Hausdorff) moment problem, which consists of finding the function $f\in L^2([0,1]^n)$ given its moments $\{f_j\}_{j\in\N_0^n}$, is ill-posed. This is related to the fact that the basis given by the monomials $\{x^j\}_{j\in\N_0^n}$ is not orthogonal.
A necessary and sufficient condition for a sequence $\{f_j\}_{j\in\N_0^n}$ to be the moments of an $L^2([0,1]^n)$ function is (see, for instance, \cite[Theorem 4.1]{AGVT04})  that
\begin{align}\label{eq:condmoments}
\sum_{k\in\N_0^n}\Bigg(\sum_{\substack{j\in\N_0^n\\ j_i\leq k_i}}C_{k_1,j_1}\dots C_{k_n,j_n} f_j\Bigg)^2<\infty,
\end{align}
where the constants $C_{m,\ell}$ are the coefficients of the Legendre polynomials, i.e. $L_{m}(t)=\displaystyle{\sum_{\ell=0}^{m} C_{m, \ell}} t^{\ell}$, which are given by
\begin{align}\label{eq:Legcoeff}
C_{m,\ell}=(2m+1)(-1)^{\ell}\frac{(m+\ell)!}{(m-\ell)!(\ell!)^2}, \quad 0\leq\ell\leq m.
\end{align}
In addition, even if there exist solutions to the moment problem, with respect to standard Sobolev spaces, the inverse of the map $f \mapsto \{f_j\}$ does not depend Lipschitz continuously on the data. Indeed,  if the solution $f\in H^1([0,1]^n)$, we have the following stability estimate  (\cite[Theorem 1]{Talenti87}  for $n=1$ and  \cite[Theorems 4.1]{AGVT04} for higher dimensions):
\begin{align}\label{eq:festmom}
\|f\|_{L^2([0,1]^n)}^2\leq e^{3.5n(N+1)}\sum_{j_1,\dots,j_n=0}^N|f_j|^2+\frac{\|\nabla f\|_{L^2([0,1]^n)}^2}{4(N+1)^2}.
\end{align}
The exponential dependence on $N$ of the first term originates from the relation between the expansion of $f$ up to degree $N$ in the Legrendre basis and in the monomial basis, which is given by a Hilbert matrix of size $N+1$ (see \cite{Talenti87} and the discussion below).

For later use, we record a translated and rescaled version of the estimate \eqref{eq:festmom}:

\begin{lem}
\label{rmk:festmom2}
Let $I=I_1\times \cdots \times I_n\subset(0,1)^n$ with $I_i$ being connected intervals in $(0,1)$ for $i=1, \dots,n$. Then an analogue to \eqref{eq:festmom} holds for any function $f\in H^1(I)$ in the sense that
\begin{align}\label{eq:festmom2}
\|f\|_{L^2(I)}^2\leq e^{C(N+1)}\sum_{j_1,\dots,j_n=0}^N|\tilde f_j|^2+\frac{\|\nabla f\|_{L^2(I)}^2}{4(N+1)^2},
\end{align}
where $C=C(|I|, n)= 6.5n-2\log |I|>0$ and $\tilde f_j=(f\chi_I)_j=\int_I f(x)x^j dx$.  
\end{lem}

\begin{proof}
The proof of \eqref{eq:festmom2} relies on rescaling and translating the bound \eqref{eq:festmom}.
Given $I=I_1\times\dots\times I_n\subset (0,1)^n$, let $A$ be a diagonal $n\times n$ matrix  and $x_0\in (0,1)^n$  such that $\{Ax+x_0: x\in (0,1)^n\}=I$. Notice that $A_{ii}=|I_i|$ for $i=1, \dots, n$.
Let $f\in H^1(I)$ and consider a translated and rescaled version of it, $F(x)=f(Ax+x_0): [0,1]^n \rightarrow \R$. We apply \eqref{eq:festmom} to $F(x)$ which yields that
\begin{align}
\label{eq:f_est_aux}
\|F\|_{L^2([0,1]^n)}^2\leq e^{3.5n(N+1)}\sum_{j_1,\dots,j_n=0}^N|F_j|^2+\frac{\|\nabla F\|_{L^2([0,1]^n)}^2}{4(N+1)^2}.
\end{align}
We now seek to return from the estimate for $F$ to an estimate for the original function $f$. To this end, we observe the following identities:
\begin{align*} 
\|F\|_{L^2((0,1)^n)}&=\frac{1}{|\det A|^{\frac{1}{2}}} \|f\|_{L^2(I)},\\
\|\nabla F\|_{L^2((0,1)^n)}&=\frac{1}{|\det A|^{\frac{1}{2}}} \|A\nabla f\|_{L^2(I)}\leq \frac{1}{|\det A|^{\frac{1}{2}}} \|\nabla f\|_{L^2(I)},\\
F_j&=\int_{(0,1)^n} F(x)x^j dx=\frac{1}{|\det A|}\int_I f(y) (A^{-1}(y-x_0))^j dy \\
&=\frac{1}{|\det A|}\sum_{\substack{0\leq k_i\leq j_i\\i=1,\dots,n}} {{j}\choose{k}}|I_1|^{-j_1}\dots |I_n|^{-j_n}(-x_0)^{j-k}\tilde f_k,
\end{align*}
where ${{j}\choose{k}}={{j_1}\choose{k_1}}\dots {{j_n}\choose{k_n}}$.
By virtue of the last identity, we have in addition
\begin{align*}
\sum_{j_1,\dots,j_n=0}^{N}|F_j|^2\leq(N+1)^n  2^{2nN}|I|^{-2(N+1)}\sum_{j_1,\dots,j_n=0}^{N}|\tilde f_j|^2\leq e^{3n(N+1)}|I|^{-2(N+1)}\sum_{j_1,\dots,j_n=0}^{N}|\tilde f_j|^2.
\end{align*}
Inserting these bounds into \eqref{eq:f_est_aux}, we therefore obtain
\begin{align*}
\|f\|_{L^2(I)}^2\leq e^{C(N+1)}\sum_{j_1,\dots,j_n=0}^N|\tilde f_j|^2+\frac{\|\nabla f\|_{L^2(I)}^2}{4(N+1)^2},
\end{align*}
which yields \eqref{eq:festmom2} with the claimed dependence in the constant $C>0$.
\end{proof}

With a similar argument as for \eqref{eq:festmom}, one further obtains the following bound:

\begin{lem}
\label{lem:moment_low}
Let $\{a_j\}_{j\in\N_0^n}\in \ell^2$ and $g(x) = \sum\limits_{j\in\N_0^n} a_j x^j\in H^1([0,1]^n)$. Then,
\begin{align}\label{eq:moment_low}
\|g\|_{L^2([0,1]^n)}^2 \geq  e^{-3.5n (N+1)} \sum\limits_{j_1,\dots, j_n=0}^N  |a_j|^2 - \frac{\|\nabla g\|_{L^2([0,1]^n)}^2}{4(N+1)^2}.
\end{align}
\end{lem}

\begin{proof}[Proof of Lemma ~\ref{lem:moment_low}]
We only sketch the argument as it follows along the same lines as the proofs in \cite{Talenti87, AGVT04}. We first consider the splitting $g(x) = h_N(x) + t_N(x)$, where $h_N(x)$ is the $L^2([0,1]^n)$ projection of $g$ onto $\{x^j: j_i\leq N, i=1,\dots, n\}$ and $t_N(x) = g(x)- h_N(x)$ is the projection onto the orthogonal complement. Expanding this into the Legendre basis functions, on the one hand, we obtain that for some $\{\lambda_j: j_i\leq N, i=1,\dots,n \} \subset \R$ we have
\begin{align*}
h_N(x) = \sum\limits_{k_1,\dots,k_n=0}^N \lambda_k L_{k_1}(x_1)\dots L_{k_n}(x_n), \quad \|h_N\|_{L^2([0,1]^n)}^2 = \sum\limits_{k_1,\dots,k_n=0}^N |\lambda_k|^2.
\end{align*}
We note that on the other hand, $h_N(x) = \sum\limits_{j_1,\dots,j_N=0}^N a_j x^j$. Hence, 
\begin{align*}
a_j = \sum\limits_{\substack{k_i=j_i\\ i=1,\dots,n}}^{N} C_{k_1,j_1}\dots  C_{k_n,j_n}\lambda_k, \ k \in \N_0^n,
\end{align*}
where $C_{k_i, j_i}$ are  given by \eqref{eq:Legcoeff}. As explained in \cite{Talenti87}, the matrix $C = \{C_{m,\ell}\}_{m,\ell \in \{0,\dots, N\}}$ is related to a Hilbert matrix $H_N = \{\frac{1}{m+\ell+1}\}_{m,\ell \in \{0,\dots, N\}}$ by the identity $C^T C = H_N^{-1}$. Since the largest singular value of the inverse of the Hilbert matrix $H_N^{-1}$ is known to be of the order $e^{3.5 (N+1)}$ (see equation (9) in \cite{Talenti87}) and the expression for $a_j$ is an $n$-fold multiplication of matrices of the form $C$, we thus obtain that
\begin{align*}
\sum\limits_{j_1,\dots,j_n=0}^N |\lambda_j|^2 \geq e^{-3.5 n (N+1)} \sum\limits_{j_1,\dots,j_n=0}^{N} |a_j|^2.
\end{align*}
Hence,
\begin{align*}
\|g\|_{L^2([0,1]^n)}^2 \geq \|h_N\|_{L^2([0,1]^n)}^2 - \|t_N\|_{L^2([0,1]^n)} ^2
\geq e^{-3.5 n (N+1)} \sum\limits_{j_1,\dots,j_n=0}^{N} |a_j|^2 - \|t_N\|_{L^2([0,1]^n)}^2 .
\end{align*}
Now, the estimate for $t_N$ follows along the exactly same lines as in \cite{AGVT04} using the ODE which is satisfied by the Legendre polynomials. Indeed, we recall that Legendre polynomials satisfy the ODE
\begin{align*}
-\big(x(1-x) L_m'(x)\big)' = m(m+1) L_m(x).
\end{align*}
In one dimension, after multiplying this ODE with $g$, recalling the definition of the coefficients $\lambda_k$ (in one dimension) and integrating by parts this entails that
\begin{align*}
\sum\limits_{k=0}^{\infty} (k+1)k|\lambda_k|^2 = \int\limits_0^1 x (1-x) |g'(x)|^2 dx.
\end{align*}
In higher dimensions this follows analogously. Recalling that $t_N(x)= \sum\limits_{k_1,\dots,k_n=N}^\infty\lambda_k L_{k_1}(x_1)\dots L_{k_2}(x_n)$ thus implies that
\begin{align*}
\|t_N\|_{L^2([0,1]^n)}^2 = 
\sum\limits_{k_1,\dots,k_n=N+1}^\infty |\lambda_k|^2 \leq \frac{1}{4 (N+1)^2}\| \nabla g \|_{L^2([0,1]^n)}^2.
\end{align*}
\end{proof}

As already indicated above, in the sequel, it will be convenient to have estimates in general intervals at our disposal. As a consequence,  we record the following translated and rescaled version of Lemma \ref{lem:moment_low}:

\begin{cor}\label{cor:moment_low2}
Let $I=I_1\times \cdots \times I_n\subset(0,1)^n$, where $I_i$ are connected intervals in $(0,1)$ for $i=1, \dots,n$. Let $\{a_j\}_{j\in\N_0^n}\in \ell^2$ and $g(x)=\sum_{j\in\N_0^n} a_j x^j\in H^1(I)$. Then, 
\begin{align*}
\|g\|_{L^2(I)}^2 \geq  e^{-C (N+1)} \sum\limits_{j_1,\dots, j_n=0}^N  |a_j|^2 - \frac{\|\nabla g\|_{L^2(I)}^2}{4(N+1)^2},
\end{align*}
where $C=C(|I|, n)=5.5n-2\log |I|>0$.
\end{cor}

\begin{proof}
The proof follows as the one of Lemma ~\ref{lem:moment_low} by now using rescaled and translated Legendre polynomials. 
For the sake of simplicity, here we present the proof just in one dimension.  

We consider $I=(x_0, x_0+\lambda)\subset(0,1)$ and the splitting $g(x)=h_N(x)+t_N(x)$ as above.
Let $\{\mathcal L_m\}_{m=0}^\infty$ be the translated Legendre polynomials  given by $\mathcal L_m(x)=\sqrt{\lambda^{-1}}L_m(\lambda^{-1}(x-x_0))$ for $x\in I$. They are orthonormal in $I$ and satisfy the ODE
\begin{align}\label{eq:ODEscLeg}
-\big((x-x_0)(\lambda+x_0-x)\mathcal L_m'(x)\big)'=m(m+1)\mathcal L_m(x).
\end{align}
Moreover, let $\mathcal C_{m,\ell}$ be the coefficients of $\mathcal L_m$, $\mathcal L_m(x)=\sum_{\ell=0}^m \mathcal C_{m,\ell}x^\ell$, which are related to the coefficients $C_{m,\ell}$ from the Legendre polynomials in the proof of Lemma \ref{lem:moment_low} according to
\begin{align}
\label{eq:coeff_new}
\mathcal C_{m,\ell}=\sum_{k=\ell}^m C_{m,k}\lambda^{-k-\frac{1}{2}}{{k}\choose{\ell}}(-x_0)^{k-\ell}.
\end{align}

Comparing the expansion of $h_N$ into the (translated and rescaled) Legendre basis and its monomial expansion, we obtain that 
\begin{align*}
h_N(x)=\sum_{k=0}^N \lambda_k \mathcal L_k(x) = \sum\limits_{j=0}^{N} a_j x^j \mbox{ with } 
a_j=\sum_{k=j}^N \mathcal C_{k,j} \lambda_k.
\end{align*}
Now, by virtue of the relation \eqref{eq:coeff_new}, we have that $\mathcal{C} = C G$, where $\mathcal{C} = \{\mathcal{C}_{jk}\}_{j,k \in \{1,\dots,N\}}$, $C = \{C_{jk}\}_{j,k \in \{1,\dots,N\}}$ and $G = \{G_{jk}\}_{j,k \in \{1,\dots,N\}}$ with $G_{k \ell} = \lambda^{- k - \frac{1}{2}}{{k}\choose{\ell}}(-x_0)^{k-\ell} $. As a consequence, we infer that
\begin{align*}
\sigma_{\max}^2(\mathcal{C}) \leq \sigma_{\max}^2(C) \sigma_{\max}^2(G) \leq e^{3.5 (N+1)} \sigma_{\max}^2(G)
\leq e^{3.5 (N+1)} 2^{2N} \lambda^{-2N - 1} .
\end{align*}
Therefore, following  \cite{Talenti87}, we arrive at
\begin{align*}
\|h_N\|_{L^2(I)}^2=\sum\limits_{j=0}^N |\lambda_j|^2 \geq e^{-5.5 (N+1)}\lambda^{2N+1} \sum\limits_{j=0}^{N} |a_j|^2.
\end{align*}
Finally, by \eqref{eq:ODEscLeg} and  $(x-x_0)(\lambda+x_0-x)\leq \frac{\lambda^2}{4}$ for $x\in I$, we obtain 
\begin{align*}
\|t_N\|_{L^2(I)}^2=\sum\limits_{j=N+1}^\infty |\lambda_j|^2 \leq \frac{\lambda^2}{4 (N+1)^2}\| \nabla g \|_{L^2(I)}^2\leq \frac{1}{4 (N+1)^2}\| \nabla g \|_{L^2(I)}^2.
\end{align*}
Combining all the estimates and noticing that $\lambda=|I|$,  the claimed result follows.
\end{proof}

\begin{lem}\label{lem:moments}
Let $ I \subseteq(0,1)^n$. There exist  positive constants $C_0$,  $C=C(n, |I|)=6.5n-2\log|I|$ such that for any $f\in H^1(I)$ and for any $\{\gamma_j\}_{j\in\N_0^n}\in \ell^2$ with  $\gamma_j>0$ we have
\begin{align}\label{eq:mom1}
\|f\|_{L^2(I)}  \leq \frac{C_0e^{C\frac{\|\nabla f\|_{L^2(I)}}{\| f\|_{L^2(I)}}}}{\displaystyle{\min_{j_1,\dots,j_n\leq N_f}\gamma_j}}\Big\|\sum_{j\in\N_0^n} \gamma_j \tilde f_j x^j\Big\|_{L^2(I)},
\end{align}
where  $\{\tilde f_j\}_{j\in\N_0^n}$ are the moments of $f\chi_I$ and
$N_f+1\simeq \frac{\|\nabla f\|_{L^2(I)}}{\| f\|_{L^2(I)}}$. Moreover, if $\gamma_j \in \ell^2$ with $\gamma_j\neq 0$ (but not necessarily signed) we have
\begin{align}\label{eq:mom2}
\begin{split}
\|f\|_{L^2(I)} &\leq \frac{C_0 e^{C\frac{\|\nabla f\|_{L^2(I)}}{\| f\|_{L^2(I)}}}}{\displaystyle{\min_{j_1,\dots,j_n\leq N_f}|\gamma_j|}} \left(\Big\|\sum_{j\in\N_0^n} \gamma_j \tilde f_j x^j\Big\|_{L^2(I)}  + \frac{1}{N_f+1} \Big\|\nabla\big(\sum_{j\in\N_0^n} \gamma_j \tilde f_j x^j\big)\Big\|_{L^2(I)} \right).
\end{split}
\end{align}

\end{lem}

\begin{proof}
If  we choose $N=N_f$ in \eqref{eq:festmom2} with
\begin{align*}
N_f+1\simeq \frac{\|\nabla f\|_{L^2(I)}}{\| f\|_{L^2(I)}},
\end{align*}
we can absorb the second term of \eqref{eq:festmom} into the left hand side to  obtain 
\begin{align}\label{eq:fest}
\|f\|_{L^2(I)}^2&\leq C_0 e^{C(N_f+1)}\sum_{j_1,\dots,j_n=0}^{N_f}|\tilde f_j|^2.
\end{align}

Let us consider the operator which associates to a function $f$ its weighted (with constants $\gamma_j>0$) moments and the associated formal adjoint operator:
\begin{align*}
\begin{array}{ccccccccc}
L_\gamma:& L^2(I)&\to& \ell^2,&\qquad&
L^*_\gamma:& \ell^2 &\to& L^2(I),\\
&f&\mapsto& \{\gamma_j\tilde f_j\}_{j\in\N_0^n},&\qquad
 &&\{a_j\}_{j\in\N_0^n} &\mapsto& {\displaystyle{\sum_{j\in\N_0^n}}} \gamma_ja_j x^j.
 \end{array}
\end{align*}
By pairing $L_\gamma f$  with the moments of $f$ we obtain
\begin{align*}
\Big|\sum_{j_1,\dots,j_n=0}^{N_f} \gamma_j|\tilde f_j|^2
\Big|&\leq |\langle L_\gamma f, \{\tilde f_j\}\rangle_{\ell^2} 
|=|\langle f, L_\gamma^*\{\tilde f_j\}\rangle_{L^2(I)}|\\
&\leq \|f\|_{L^2(I)}\|L_\gamma^*\{\tilde f_j\}\|_{L^2(I)}. 
\end{align*}
Assuming that $\gamma_j>0$ and using  \eqref{eq:fest}, we infer
\begin{align*}
\sum_{j_1,\dots,j_n=0}^{N_f} |\tilde f_j|^2&
\leq \frac{1}{\displaystyle{\min_{j_1,\dots,j_n\leq N_f}\gamma_j}}\sum_{j_1,\dots,j_n=0}^{N_f} \gamma_j|\tilde f_j|^2
\leq \frac{1}{\displaystyle{\min_{j_1,\dots,j_n\leq N_f}\gamma_j}}\|f\|_{L^2(I)} \|L_{\gamma}^{\ast} {\tilde f_j}\|_{L^2(I)}
\\
&\leq \frac{1}{\displaystyle{\min_{j_1,\dots,j_n\leq N_f}\gamma_j}} \left(C_0 e^{C(N_f+1)}\sum_{j_1,\dots,j_n=0}^{N_f}|\tilde f_j|^2\right)^{\frac{1}{2}}\Big\|\sum_{j\in\N_0^n} \gamma_j \tilde f_j x^j\Big\|_{L^2(I)},
\end{align*}
which leads to
\begin{align}\label{eq:fjest2}
\sum_{j_1,\dots,j_n=0}^{N_f} |\tilde f_j|^2&
\leq  \frac{C_0e^{C(N_f+1)}}{\displaystyle{\min_{j_1,\dots,j_n\leq N_f}\gamma_j^2}}\Big\|\sum_{j\in\N_0^n} \gamma_j \tilde f_j x^j\Big\|_{L^2(I)}^2.
\end{align}

Combining  \eqref{eq:fest} and \eqref{eq:fjest2} we obtain \eqref{eq:mom1}. 

Let us now consider  $g(x) = \sum\limits_{j=0}^{\infty} \gamma_j \tilde f_j x^j$ for general, non-zero, but not necessarily signed values of $\gamma_j$. By Corollary ~\ref{cor:moment_low2}  with $N=N_f$,
\begin{align*}
\|g\|_{L^2(I)}^2 \geq e^{- C (N_f+1)}\sum\limits_{j=1}^{N_f}|\gamma_j \tilde f_j|^2 - \frac{1}{4(N_f+1)^{2}}\|\nabla g\|_{L^2(I)}^2.
\end{align*}
Then, if $\gamma_j \neq 0$ for $j_i\leq N_f, i=1,\dots,n,$ we simply bound 
\begin{align*}
\sum\limits_{j_1,\dots,j_n=0}^{N_f} |\tilde f_j|^2
&\leq \frac{1}{\min\limits_{j_1,\dots, j_N\leq  N_f} |\gamma_j|^2}\sum\limits_{j_1,\dots, j_N=0}^{N_f}|\gamma_j \tilde f_j|^2
\\&\leq  \frac{C_0e^{C(N_f+1)}}{\min\limits_{j_1,\dots, j_N\leq  N_f} |\gamma_j|^2}\left(\|g\|_{L^2(I)}^2+\frac{1}{4(N_f+1)^{2}}\|\nabla g\|_{L^2(I)}^2\right),
\end{align*}
which combined with \eqref{eq:fest} leads to \eqref{eq:mom2}.
\end{proof}

\subsection{Analytic continuation}
In this section we recall results that will be used later to propagate the smallness of a function.

As a first auxiliary result, we state a slight modification of the results in \cite{Vessella99, AE13}. 

\begin{lem}\label{lem:analytcont}
Let $\Omega$ be a bounded, connected, open set in $\R^n$  and $\omega\subset \Omega$ with $|\omega|>0$.
Let $g$ be an analytic function in  $\Omega$ satisfying for any $k\in\N_0^n$
\begin{align}\label{eq:analytcontcond}
\big|\partial_x^k g(x)\big|\leq M\rho^{-|k|}k! \quad \mbox{for  }x\in\Omega
\end{align} 
for some $\rho$, $M>0$.
Then, there exist constants $C>0$ and $\theta\in (0,1)$ depending on $\frac{|\omega|}{|\Omega|}, |\omega|$ and $\rho$ such that 
\begin{align}
\label{eq:analytcontsup}
\sup_\Omega |g|&\leq CM^{1-\theta}\|g\|_{L^2(\omega)}^\theta.
\end{align}
\end{lem}

\begin{proof} 
We rely on the quantitative analytic continuation results from \cite{Vessella99, AE13}. There it is shown that under the hypotheses of the statement 
there exist $C'>0$ and $\theta\in(0,1)$ depending on $\frac{|\omega|}{|\Omega|}$ and $\rho$ such that 
\begin{align*}
\sup_\Omega|g|\leq C' M^{1-\theta}\left(\frac{\|g\|_{L^1(\omega)}}{|\omega|}\right)^\theta.
\end{align*}
By H\"older's inequality we bound the $L^1$-norm by the $L^2$-norm leading to
\begin{align*}
\sup_\Omega|g|
&\leq C'M^{1-\theta}\left(\frac{\|g\|_{L^2(\omega)}}{|\omega|^{\frac{1}{2}}}\right)^\theta
=C M^{1-\theta}\|g\|_{L^2(\omega)}^\theta.
\end{align*}
\end{proof}

\begin{rmk}
We will also use the estimate \eqref{eq:analytcontsup} in the form
\begin{align}
\label{eq:analytcontL2}
\|g\|_{L^2(\Omega)}&\leq CM^{1-\theta}\|g\|_{L^2(\omega)}^\theta,
\end{align}
which follows after applying H\"older's inequality on the left hand side of \eqref{eq:analytcontsup}.
\end{rmk}

The following auxiliary results are for complex analytic functions.

\begin{lem}\cite[pages 558-559]{J60}
\label{lem:analytcontJohn}
Let $J$ be a bounded, connected, open interval in $\R$ and let $g$ be  a function which is complex analytic in $2J\times(0,|J|) \subset \R^2$ (where $\R^2$ is identified canonically with $\C$) and continuous in its closure.  Then  there exists $\theta\in(0,1)$ such that 
\begin{align*}
\sup_{J\times[0, \frac{|J|}{2}]} |g| \leq \big(\sup_{2J} |g|\big)^\theta \big(\sup_{2J\times[0,|J|]} |g|\big)^{1-\theta}.
\end{align*}
\end{lem}
An $L^2$-version of this will be discussed in Lemma ~\ref{lem:John}.

Applying Lemma \ref{lem:analytcontJohn} to each variable separately an analogous result for functions which are analytic in several variables can be obtained: 

\begin{cor}\label{cor:analytcontJohn}
Let $g$ be  a function which is continuous in $B^\pm_{4\sqrt{n}r} \subset (\C_{\pm})^n$ and (complex) analytic in its interior.  Then  there exists $\theta\in(0,1)$ such that 
\begin{align*}
\sup_{B_r^\pm} |g| \leq \big(\sup_{(-2r,2r)^n} |g|\big)^\theta \big(\sup_{B_{4\sqrt{n}r}^\pm} |g|\big)^{1-\theta}.
\end{align*}
\end{cor}

\begin{proof}
Let us consider without loss of generality the case in $(\C_+)^n$.
The result follows after applying  Lemma ~\ref{lem:analytcontJohn} iteratively in each variable with $J=(-r,r)$ and noticing that
\begin{align*}
B_{r}^+\subset \big((-r,r)\times [0,r]\big)^n,\qquad
B_{4\sqrt{n}r}^+\supset \big((-2r,2r)\times [0,2r]\big)^n.
\end{align*}
\end{proof}

\section{Quantitative Unique Continuation through Stability of the Moment Problem, One Dimension}
\label{sec:moments1d}

In this section we present our proof of Theorem ~\ref{thm:qucHilbert_a} and thus provide an alternative proof to the ones from \cite{APS14, Rue17}. Moreover, we present new quantitative stability results for the modified Hilbert transform -- for which currently no Carleman estimates are available -- and for the inverse of the fractional Laplacian.

Given two  open, bounded subsets $I, J\subset\R$ with $\overline{I}\cap \overline{J} = \emptyset$, we can assume without loss of generality that $I\Subset (0,1)$ and $J\Subset (1, +\infty)$.  
Indeed, suppose that $J$ is located to the right of $I$, which can always be ensured, possibly after a suitable reflection. Let $I=(a, b)$ and  $d=\min\{\frac{\dist (I, J)}{|I|}, \frac{1}{2}\}>0$. Then, the change of variables
\begin{align*}
x'=\frac{1}{2} \Big(\frac{x-a}{b-a}+1-\frac{d}{2}\Big)
\end{align*}
transforms $I$ to $I'=(\frac{1}{2}-\frac{d}{4}, 1-\frac{d}{4})\Subset(0,1)$, and since $J\subset \big(b+d(b-a), \infty\big)$ then $J'\subset\big(1+\frac{d}{4},\infty\big)$.
We will make use of this fact throughout this section.

\subsection{Hilbert transform}
\label{sec:Hilbert}
In this section we prove Theorem ~\ref{thm:qucHilbert_a}, where $Hf$ is the Hilbert transform  given by 
\begin{align*}
Hf(x)=\mathcal F^{-1}\big(-i\sign(\cdot)\mathcal F f(\cdot)\big)(x)=\frac{1}{\pi}\textrm{p.v.}\int_{\R}\frac{f(y)}{x-y}dy,
\end{align*}
where the notation $\text{p.v.}$ denotes that the integral must be understood in the principal value sense.
The Hilbert transform plays an important role, for instance, in the study of inverse problems originating from medical imaging \cite{Natterer, DNCK06}. 
Its logarithmic stability has previously been established by different methods in \cite{APS14} (using Slepian's miracle \cite{SP61} in the form that the Hilbert transform commutes with a certain second order differential operator) and, subsequently, in \cite{Rue17} (viewing the Hilbert transform as boundary operator associated with an harmonic extension operator and using Carleman estimates for this). We refer to \cite{APS96} and \cite{T51} for some early works related to the stability of (inverting) truncated Hilbert type transforms.

\begin{proof}[Proof of Theorem ~\ref{thm:qucHilbert_a}]
Let us suppose without loss of generality that $I\Subset(0,1)$ and $J\Subset(1,\infty)$ and let $f\in C^\infty_c(I)$.  If we consider $x>1$, the kernel of the Hilbert transform is analytic. We expand it in $y$ and use the definition of the moments in \eqref{eq:defmom} to infer
\begin{align}\label{eq:Hexp}
\begin{split}
Hf(x)&=\frac{1}{\pi}\int_{\R}\frac{f(y)}{x-y}dy
=\frac{1}{\pi}\int_{\R}\frac{1}{x}\frac{f(y)}{1-\frac{y}{x}}dy
=\frac{1}{\pi x}\int_{\R}f(y)\sum_{j=0}^\infty\left(\frac{y}{x}\right)^jdy\\
&=\frac{1}{\pi}\sum_{j=0}^\infty x^{-1-j}\int_{\R}y^j f(y)dy
=\frac{1}{\pi}\sum_{j=0}^\infty f_j x^{-j-1}.
\end{split}
\end{align}

By \eqref{eq:mom1}  in Lemma ~\ref{lem:moments} with  $\gamma_j=1$ and since $\tilde f_j=f_j$ (which follows from the fact that $f\in C_c^{\infty}(I)$) we have
\begin{align}
\label{eq:HT_estimate}
\|f\|_{L^2(I)}
\leq Ce^{C(N_f+1)}\Big\|\sum_{j=0}^\infty f_jz^j\Big\|_{L^2(I)},
\end{align}
where $N_f+1\simeq \frac{\|f'\|_{L^2(I)}}{\| f\|_{L^2(I)} } \leq \frac{\|f\|_{H^1(I)}}{\| f\|_{L^2(I)} } $. 
We seek to write the  norm on the right hand side of \eqref{eq:HT_estimate} in terms of $Hf$. By the change of variables $z=x^{-1}$ we obtain
\begin{align*}
\Big\|\sum_{j=0}^\infty f_jz^j\Big\|_{L^2(I)}=\Big\|\sum_{j=0}^\infty f_jx^{-j-1}\Big\|_{L^2(I^{-1})}=\pi \|Hf\|_{L^2(I^{-1})},
\end{align*}
where $ I^{-1}\Subset(1,\infty)$.
Therefore,
\begin{align}\label{eq:fHf1}
\|f\|_{L^2(I)}
&\leq C e^{C(N_f+1)}\|Hf\|_{L^2(I^{-1})}.
\end{align}

If $I^{-1}\subseteq J$, the desired result follows directly.
Otherwise, i.e. if  $I^{-1}\nsubseteq J$, we need to propagate the information of $Hf$ in $J$ into $I^{-1}$.
Since $Hf$ is (real) analytic in $(1,\infty)$, we seek to apply Lemma ~\ref{lem:analytcont} with $\Omega=\text{conv}(I^{-1}\cup J)$  and  $\omega=J$. (We remark that the openness of $J$ is not required, but only that $J$ has positive measure, as it is pointed out in Remark ~\ref{rmk:Jmeasure}.) To that end, we observe that for any $k\in\N_0$ and $x\in\Omega$, we have
\begin{align}\label{eq:HfM}
\Big|\frac{d^k}{dx^k} Hf(x)\Big|=\frac{k!}{\pi}\left|\int_I \frac{f(y)}{(x-y)^{1+k}}dy\right|\leq k! \rho^{-1-k} \pi^{-1}\|f\|_{L^2(I)}
\end{align}
with $\rho=\dist(I, \Omega)$. Then, the hypothesis \eqref{eq:analytcontcond} holds with $M=(\rho\pi)^{-1}\|f\|_{L^2(I)}$. The application of Lemma ~\ref{lem:analytcont} hence yields  the existence of constants $C>0$ and $\theta\in(0,1)$ depending on $\Omega$, $J$ and $\rho$ (and therefore only on $I$ and $J$) such that \begin{align*}
\|Hf\|_{L^2(I^{-1})}\leq  C \|f\|_{L^2(I)}^{1-\theta}\|Hf\|_{L^2(J)}^\theta.
\end{align*}
The combination of this estimate with \eqref{eq:fHf1} implies the desired result:
\begin{align*}
\|f\|_{L^2(I)}
&\leq Ce^{\frac{C}{\theta}(N_f+1)}\|Hf\|_{L^2(J)}.
\end{align*}
This concludes the proof of Theorem \ref{thm:qucHilbert_a}.
\end{proof}

\begin{rmk}
\label{rmk:relax_support_HT}
Examining the proof of Theorem \ref{thm:qucHilbert_a}, we note that if $f\in H^1(I)$, we have $H(\chi_I f)(x)=\frac{1}{\pi}\sum_{j=0}^\infty \tilde f_j x^{-j-1}$ and that the proof for \eqref{eq:truncatedHT} can thus be carried out along the same lines as above with the only difference of always replacing $f_j$ by $\tilde f_j$.
\end{rmk}

\subsection{The modified Hilbert transform}
\label{sec:mod_Hilb}
In this section we consider a perturbation of the Hilbert transform which, for instance, naturally arises in fluid mechanics.  
For $\delta$ being a positive number, we define the modified Hilbert transform $H_\delta$ by
\begin{align*}
H_\delta f(x)=\mathcal F^{-1}\left(-i\coth\big(\frac{\cdot}{2\delta} \big)\mathcal F f(\cdot)\right)(x)=\delta \text{ p.v.}\int_{\R}\coth\big(\pi\delta(x-y)\big)f(y)dy.
\end{align*}
Formally, in the limit $\delta\to 0$, the Hilbert transform is recovered.

The modified Hilbert transform $H_{\delta} f$ appears in the intermediate long wave equation \cite{Joseph77}, which describes long internal gravity waves in a stratified fluid with finite depth represented by the parameter $\delta$. Qualitative unique continuation results had previously been considered in \cite{KPV20} (see also \cite{KPPV20a} for applications of the qualitative uniqueness ideas to other nonlinear dispersive equations). 

We begin by proving a logarithmic stability estimate for the modified Hilbert transform by reducing it to a perturbation of the estimate for the Hilbert transform from the previous section. In order to carry out this strategy we require additional conditions on $f$. In Proposition ~\ref{prop:qucTHilbert1} we dispose of these, however at the expense of dealing with $H^1(J)$ instead of $L^2(J)$-measurements.

\begin{prop}\label{prop:qucTHilbert}
Let $I$ and $J$ be two open, bounded, non-empty subsets of $\R$ such that  $\overline I\cap\overline J=\emptyset$, let $\delta>0$ and let $f\in C^\infty_c(I)$.
Suppose that one of the following conditions holds:
\begin{enumerate}[label=(\roman*)]
\item $f$ has zero mean, i.e. $\int_I f(x)dx=0$,
\item there exists a constant $C'$ independent of $f$ such that $|\int_I f(x)dx|\leq C' \|H_\delta f\|_{L^2(J)}$, which in particular holds if $f$ has a constant sign,
\item $\delta<\delta_0$ for some $\delta_0$ depending on $I$ and $\frac{\|f\|_{H^1(I)}}{\|f\|_{L^2(I)}}$ (see \eqref{eq:relation}).
\end{enumerate}
Then there exists a constant $C=C(I,J,\delta)>0$ such that
\begin{align}\label{eq:qucTHilbert}
\|f\|_{L^2(I)}
&\leq C e^{C\frac{\|f'\|_{L^2(I)}}{\| f\|_{L^2(I)}}}\|H_\delta f\|_{L^2(J)}.
\end{align}
\end{prop}

We seek to argue similarly as in the case of the Hilbert transform in the previous section. In order to view  $H_{\delta}f$ as a perturbation of the Hilbert transform, we introduce the variables
\begin{align}\label{eq:txx}
\tx=\frac{e^{2\pi\delta x}-1}{2\pi\delta}, \qquad \ty=\frac{e^{2\pi\delta y}-1}{2\pi\delta},
\end{align}
which coincide with $x$ and $y$ in the limit $\delta\to 0$.
Then, the kernel of the modified Hilbert transform, 
\begin{align*}
k_\delta(x, y)=\delta\coth\big(\pi\delta(x-y)\big)
=\delta\frac{e^{2\pi\delta x}+{e^{2\pi\delta y}}}{e^{2\pi\delta x}-{e^{2\pi\delta y}}},
\end{align*}
in the new variables corresponds to
\begin{align*}
\tilde k_\delta(\tx,\ty)=\delta\frac{\tx+\ty+\frac{1}{\pi\delta}}{\tx-\ty}
=\left(\frac{1}{\pi}+2\delta\tx\right)\frac{1}{\tx-\ty}-\delta, 
\end{align*}
which in the limit $\delta\to 0$ coincides with the kernel  of the Hilbert transform $\frac{1}{\pi}\frac{1}{x-y}$.

Let thus $\tilde H_\delta$ denote the operator associated with the kernel $\tilde k_\delta(\tx,\ty)$, i.e.,
\begin{align}\label{eq;deftHd}
\tilde H_\delta F(\tx)
&=\text{p.v.}\int_\R F(\ty)\left(\left(\frac{1}{\pi}+2\delta\tx\right)\frac{1}{\tx-\ty}-\delta\right) d\ty.
\end{align}
We have the following stability result for it, which in the limit $\delta \rightarrow 0$ recovers the result of Theorem ~\ref{thm:qucHilbert_a}:

\begin{lem}\label{lem:qucTtildeH}
Let  $\tilde I, \tilde J \Subset(0,\infty)$ be  two  open, bounded, non-empty subsets such that $\tilde I\cap\tilde J=\emptyset$ with $\tilde J$ located to the right of $\tilde I$  and let $\delta>0$. Let {$\mu=\max\{\frac{\sup \tilde I+\inf \tilde J}{2},1\}$}.
Then there exist a universal constant $C_0>0$ and   constants  $\tilde C=\tilde C(\mu^{-1} |\tilde I|)>0$, $\nu=\nu(\tilde I, \tilde J)>1$ and  $C=C(\tilde I, \tilde J, \delta)>0$  such that for any $F\in C^\infty_c(\tilde I)$ we have
\begin{align*}
\|F\|_{L^2(\tilde I)}
&\leq C_0^\nu e^{\tilde C\nu \mu \frac{\|F'\|_{L^2(\tilde I)}}{\|F\|_{L^2(\tilde I)}}}\left(C\|\tilde H_\delta F \|_{L^2(\tilde J)} +\delta^{\frac{\nu}{2}}\mu^{\frac{\nu-1}{2}}|F_0|\right),
\end{align*}
where $F_0=\int_\R F(\tx)d\tx$.
\end{lem}

\begin{proof}
Let assume for the moment that $\tilde I\Subset (0,1)$ and $\tilde J\Subset(1,\infty)$ and 
let $F\in C^\infty_c(\tilde I)$. 
Then if $\tx>1$, we have, similarly as in \eqref{eq:Hexp}, the following expansion:
\begin{align}\label{eq:tHdexp}
\begin{split}
\tilde H_\delta F(\tx)
&=\int_\R F(\ty)\left(\frac{\pi^{-1}+2\delta \tx}{\tx}\sum_{j=0}^\infty \left(\frac{\ty}{\tx}\right)^j-\delta\right) d\ty
\\
&=\Big(\frac{1}{\pi}+2\delta \tx\Big)\sum_{j=0}^\infty F_j \tx^{-j-1}-\delta F_0.
\end{split}
\end{align}

By  Lemma ~\ref{lem:moments} with $\gamma_j=1$, 
\begin{align*}
\|F\|_{L^2(\tilde I)}
&\leq C_0e^{\tilde C (N_F+1)}\Big\|\sum_{j=0}^\infty F_jz^j\Big\|_{L^2(\tilde I)},
\end{align*}
where $N_F+1\simeq\frac{\|F'\|_{L^2(\tilde I)}}{\|F\|_{L^2(\tilde I)}}$ and $\tilde C$ depends only on $|\tilde I|$.
Changing variables according to $z=\tx^{-1}$ and comparing with  \eqref{eq:tHdexp} we deduce 
\begin{align*}
\Big\|\sum_{j=0}^\infty F_jz^j\Big\|_{L^2(\tilde I)}
\leq \Big\|\sum_{j=0}^\infty F_j\tx^{-j-1}\Big\|_{L^2(\tilde I^{-1})}
&\leq \Big\|\frac{\tilde H_\delta F(\cdot)}{\pi^{-1}+2\delta\tx}\Big\|_{L^2(\tilde I^{-1})} +\delta |F_0|\Big\|\frac{1}{{\pi^{-1}+2\delta\tx}}\Big\|_{L^2(\tilde I^{-1})}
\\&
\leq \pi \|\tilde H_\delta F(\cdot) \|_{L^2(\tilde I^{-1})} + \left(\frac{\pi\delta}{2}\right)^{\frac{1}{2}}|F_0|, 
\end{align*}
with  $\tilde I^{-1}\Subset(1,\infty)$ and where  we have bounded the last norm by observing that
\begin{align*}
\int_{\tilde I^{-1}} \frac{1}{(\pi^{-1}+2\delta\tx)^2}d\tx
\leq\int_{1}^\infty \frac{1}{(\pi^{-1}+2\delta\tx)^2}d\tx
=\frac{1}{2\delta}\frac{1}{{\pi^{-1}+2\delta}}\leq \frac{\pi}{2\delta}.
\end{align*}

At this point, we are interested in estimating the norm $\|\tilde H_\delta F(\cdot) \|_{L^2(\tilde I^{-1})}$ in terms of $\|\tilde H_\delta F(\cdot) \|_{L^2(\tilde J)}$.
If $\tilde I^{-1}\subseteq \tilde J$, this  follows immediately. Otherwise, we apply Lemma ~\ref{lem:analytcont} with $\Omega=\text{conv}(\tilde I^{-1}\cup\tilde J)$, in which $\tilde H_\delta F$ is (real) analytic, and $\omega = \tilde{J}$. For any $\tx\in\Omega$ we have 
\begin{align}\label{eq:tHFM}
\begin{split}
|\big(\tilde H_\delta F\big)(\tilde x)|
&=\left| \int_{\tilde I}\left(\frac{\pi^{-1}+2\delta \tilde{y}}{\tilde{x}- \tilde{y}}+\delta\right) F(\ty)d\ty \right|
\leq \left(\frac{\pi^{-1}+2\delta}{\rho}+\delta \right)\|F\|_{L^2(\tilde I)},
\\
\Big|\frac{d^k}{d\tx^k}\big(\tilde H_\delta F\big)(\tilde x)\Big|
&=k! \left|\int_{\tilde I}\frac{\pi^{-1}+2\delta \tilde{y}}{(\tilde{x}-\tilde{y})^{1+k}}F(\ty)d\ty \right|
\leq k!\left(\frac{\pi^{-1}+2\delta}{\rho^{1+k}} \right)\|F\|_{L^2(\tilde I)}, \quad k\in\N,
\end{split}
\end{align}
with $\rho=\dist(\tilde I, \Omega)>0$ (which follows by the assumption that $\tilde{I}\Subset (0,1)$, $\tilde{J} \Subset (1,\infty)$).  Then \eqref{eq:analytcontcond} holds with $M=\left(\frac{\pi^{-1}+2\delta}{\rho}+\delta \right)\|F\|_{L^2(\tilde I)}$ and thus
there exist $C>0$ and $\theta\in(0,1)$ depending on $\tilde I$ and $\tilde J$ such that 
\begin{align}\label{eq:tHFancont}
\|\tilde H_\delta F\|_{L^2(\tilde I^{-1})}
\leq C \left(\frac{\pi^{-1}+2\delta}{\rho}+\delta \right)^{1-\theta}\|F\|_{L^2(\tilde I)}^{1-\theta} \|\tilde H_\delta F\|_{L^2(\tilde J)}^\theta.
\end{align}
Finally, notice also that $|F_0|\leq \|F\|_{L^2(\tilde I)}$. Combining all these considerations we obtain
\begin{align*}
\|F\|_{L^2(\tilde I)}\leq C_0e^{\tilde C (N_F+1)}\left(C\|\tilde H_\delta F\|_{L^2(\tilde J)}^\theta+\delta^{\frac{1}{2}}|F_0|^\theta\right)\|F\|_{L^2(\tilde I)}^{1-\theta},
\end{align*}
where $C_0$ is a universal constant and $C$ depends on $\tilde I, \tilde J$ and $\delta$. This estimate finally implies
\begin{align}\label{eq:qucTtildeHint}
\|F\|_{L^2(\tilde I)}
&\leq C_0^\nu e^{\tilde C\nu \frac{\|F'\|_{L^2(\tilde I)}}{\|F\|_{L^2(\tilde I)}}}\left(C\|\tilde H_\delta F \|_{L^2(\tilde J)} +\delta^{\frac{\nu}{2}}|F_0|\right),
\end{align}
where $\nu=\theta^{-1}$, which gives us the desired result since {$\mu\geq 1$}. 

In general, for any $\tilde I$ and $\tilde J$, let us consider $\lambda=2(\sup \tilde I+\inf \tilde J)^{-1}$ and the changes of variables $\tx'=\lambda\tx$.
Then $\tilde I$, $\tilde J$  are transformed to $\tilde I'$, $\tilde J'$ satisfying $\tilde I'\Subset (0,1)$,  $\tilde J'\Subset(1,\infty)$.
Let $\tilde F(\tx')=F(\frac{\tx'}{\lambda})$, for which \eqref{eq:qucTtildeHint} holds, and for which
\begin{align*}
\tilde H_\delta \tilde F(\tx')&=\int_{\tilde I'}\delta\frac{\tx'+\ty'+\frac{1}{\pi\delta}}{\tx'-\ty'} \tilde F(\ty')d\ty'
=\left(\int_{\tilde I} (\lambda\delta)\frac{\tx+\ty+\frac{1}{\pi\lambda\delta}}{\tx-\ty}  F(\ty)d\ty\right)_{\tx=\frac{\tx'}{\lambda}}
=(\tilde H_{\lambda\delta}F)\Big(\frac{\tx'}{\lambda}\Big).
\end{align*}
Therefore,
\begin{align*}
\|F\|_{L^2(\tilde I)}
&\leq C_0^\nu e^{\tilde C\nu\lambda^{-1} \frac{\|F'\|_{L^2(\tilde I)}}{\|F\|_{L^2(\tilde I)}}}\left(C\|\tilde H_{\lambda\delta} F \|_{L^2(\tilde J)} +\delta^{\frac{\nu}{2}}\lambda^{\frac{1}{2}}|F_0|\right),
\end{align*}
where $\tilde C$ depends on $|\tilde I'|=\lambda|\tilde I|$. Renaming $\lambda \delta$ as $\delta$ and $\mu=\lambda^{-1}$, the desired result follows.
\end{proof}

With Lemma ~\ref{lem:qucTtildeH} in hand, we now address the proof of 
 Proposition ~\ref{prop:qucTHilbert} by reducing the statement for the function $f$ from Proposition ~\ref{prop:qucTHilbert} to a suitable function $F$ (depending on $f$) as in Lemma ~\ref{lem:qucTtildeH}.

\begin{proof}[Proof of Proposition ~\ref{prop:qucTHilbert}]
Let us assume without loss of generality that $I$ and $J$ are located in the positive real half-line with $J$ located to the right of $I$.
Let $f\in C^\infty_c(I)$ and  consider the function
\begin{align}\label{eq:Ff}
F(\tx)=\frac{1}{1+2\pi\delta\tx} f\Big(\frac{1}{2\pi\delta}\log(1+2\pi\delta \tx)\Big),
\end{align}
which is compactly supported in $\tilde I=\{\frac{e^{2\pi\delta x}-1}{2\pi\delta}: x\in I\}$ and which satisfies
\begin{align}\label{eq:HdFf}
\tilde H_\delta F(\tx)=(H_\delta f)\Big(\frac{1}{2\pi\delta}\log(1+2\pi\delta \tx)\Big).
\end{align}
 Let $\tilde J=\{\frac{e^{2\pi\delta x}-1}{2\pi\delta}: x\in J\}\neq \emptyset$.  Notice that $\tilde I,\tilde J$  satisfy the hypothesis of Lemma ~\ref{lem:qucTtildeH} and therefore
\begin{align}\label{eq:qucF}
\|F\|_{L^2(\tilde I)}
&\leq C_0^\nu e^{\tilde C\nu \mu\frac{\|F'\|_{L^2(\tilde I)}}{\|F\|_{L^2(\tilde I)}}}\left(C\|\tilde H_\delta F \|_{L^2(\tilde J)} +\delta^{\frac{\nu}{2}}\mu^{\frac{\nu-1}{2}}|F_0|\right),
\end{align}
where $\mu=\max\{\frac{\sup \tilde I+\inf \tilde J}{2},1\}$, $\tilde C>0$ depends on $\mu^{-1} |\tilde I|$, and $C>0$ and $\nu>1$ depend on $\tilde I$ and $\tilde J$.

Now we seek to estimate the norms which appear in \eqref{eq:qucF} in terms of $f$. Let $b=\sup I>0$.
By \eqref{eq:Ff}, \eqref{eq:HdFf} and \eqref{eq:txx}  we have
\begin{align}\label{eq:estimatesFf}
\begin{split}
\|F\|_{L^2(\tilde I)}&=\|f(\cdot) e^{-\pi\delta x}\|_{L^2(I)}
\geq e^{-\pi\delta b}\|f\|_{L^2(I)},
\\
\|F'\|_{L^2(\tilde I)}&=\|(f'(\cdot)-2\pi\delta f(\cdot)) e^{-3\pi\delta x}\|_{L^2(I)}\leq \|f'\|_{L^2(I)}+2\pi\delta \|f\|_{L^2(I)} ,
\\
\|\tilde H_\delta F\|_{L^2(\tilde J)}&=\|(H_\delta f)(\cdot) e^{\pi\delta x}\|_{L^2(J)}\leq  e^{\pi\delta\sup J} \|H_\delta f \|_{L^2(J)}.
\end{split}
\end{align}
Moreover, by construction $F_0=f_0=\int_I f(x)dx$. Indeed,
\begin{align*}
F_0&=\int_{\tilde I} F(\tx) d\tx=\int_{\tilde I} \frac{1}{1+2\pi\delta\tx} f\Big(\frac{1}{2\pi\delta}\log(1+2\pi\delta \tx)\Big)d\tx\\
&=\int_I \frac{1}{e^{2\pi\delta x}} f(x) e^{2\pi\delta x}dx=\int_I f(x)dx=f_0,
\end{align*}
where we have applied the change of variables \eqref{eq:txx}.
Inserting all these bounds into \eqref{eq:qucF}, we arrive at
\begin{align}\label{eq:estfHdf0}
\|f\|_{L^2(I)}
&\leq C_0^\nu e^{\pi\delta b}e^{\tilde C\nu\mu  e^{\pi\delta b}\left(\frac{\|f'\|_{L^2(I)}}{\|f\|_{L^2(I)}}+2\pi\delta\right)}\Big(C\|H_\delta f\|_{L^2(J)}+\delta^{\frac{\nu}{2}}\mu^{\frac{\nu-1}{2}} |f_0|\Big).
\end{align}
Finally, we discuss the conditions from the proposition under which the
estimate  \eqref{eq:qucTHilbert} can be deduced:

\begin{enumerate}[label=\textit{(\roman*)}]
\item If $f_0=0$, then \eqref{eq:qucTHilbert} immediately follows from \eqref{eq:estfHdf0}.

\item If $|f_0|\leq  C'\|H_\delta f\|_{L^2(J)}$, where $ C'$ may depend on $\delta$, $I$ and $J$ but not on $f$, then
\eqref{eq:qucTHilbert} is also achieved directly.
This happens in particular if $f$ has a constant sign, since 
in this case
\begin{align*}
|f_0|
=\int_I |f(y)|dy
&\leq \frac{1}{ \coth\big(\pi\delta\dist(I,J)\big) }\int_I\coth\big(\pi\delta(x-y)\big)|f(y)|dy\\
&\leq \frac{1}{\delta |\coth\big(\pi\delta\dist(I,J)\big)| }|H_\delta f(x)|, \quad x\in J.
\end{align*}
Moreover, if $f$ has a constant sign, then $F$ also has a constant sign and thus by \eqref{eq:tHdexp} for $\tx\in\tilde J$ we have
\begin{align*}
\delta |F_0|\leq \left(\frac{1}{\pi}+\delta \tx\right)\Big|\sum_{j=0}^\infty F_j\tx^{-j-1}\Big|\leq |\tilde H_\delta F(\tx)|.
\end{align*}
As a consequence, also in the proof of Lemma ~\ref{lem:qucTtildeH} we may ignore the contribution depending on $F_0$ in the case that $f$ (and $F$) has a constant sign.

\item Last but not least, since $|f_0|\leq |I|^{\frac{1}{2}}\|f\|_{L^2(I)}$, we can always absorb the right hand side contribution resulting from the presence of $f_0$, i.e. the contribution 
\begin{align}\label{eq:THcond}
 \left(C_0e^{\tilde C\mu  e^{\pi\delta b}\left(\frac{\|f'\|_{L^2(I)}}{\|f\|_{L^2(I)}}+2\pi\delta\right)}(\mu\delta)^{\frac{1}{2}} \right)^{\nu} \mu^{-\frac{1}{2}}|I|^{\frac{1}{2}} e^{\pi\delta b}\|f\|_{L^2(I)},
\end{align}
into the left hand side of \eqref{eq:estfHdf0} if $\delta \in (0,1)$ is chosen sufficiently small depending on $\frac{\|f\|_{H^1(I)}}{\|f\|_{L^2(I)}}$. 
More precisely, recalling that $\mu \geq 1$, we choose $\delta$ to be small enough such that it satisfies the following estimate
\begin{align}
\label{eq:relation}
 C_0 e^{\pi\delta b} e^{\tilde C \mu  e^{\pi\delta b}\left(\frac{\|f'\|_{L^2(I)}}{\|f\|_{L^2(I)}}+2\pi\delta\right)}\delta^{\frac{1}{2}}
\leq \frac{1}{2\max\{1,|I|^{\frac{1}{2}}\}}.
\end{align}
This is always possible for small enough values of $\delta>0$, since, keeping track of the constants, we have
\begin{align*}
1\leq \mu&<{\max\{\inf \tilde J,1\}}= {\max\left\{1,\frac{e^{2\pi\delta \inf J}-1}{2\pi\delta}\right\}} = \max\left\{ 1 , \inf J + O(\delta) \right\},\\
\tilde C&=6.5-\log(\mu^{-1}|\tilde I|) \leq 6.5 + \log(\mu)- \log(|\tilde{I}|)\\
&\leq 6.5 + \log(\max\left\{ 1 , \inf J + O(\delta) \right\})+ |\log(|I| + O(\delta) )| ,
\end{align*}
which is uniformly bounded as $\delta \rightarrow 0$.
Then, since the right hand side of \eqref{eq:relation} is bounded by $\frac{1}{2}$ and $\nu>1$, \eqref{eq:THcond} is bounded by $\frac{1}{2}\|f\|_{L^2(I)}$ and  therefore \eqref{eq:qucTHilbert} holds.

\end{enumerate}
\end{proof}

Exchanging the $L^2(J)$ measurements of $H_{\delta}(f)$ by $H^1(J)$ measurements of $H_{\delta}(f)$, we can provide an unconditional stability estimate for the modified Hilbert transform. 

\begin{prop}\label{prop:qucTHilbert1}
Let $I$ and $J$ be two  open, bounded, non-empty subsets of $\R$ such that $\overline I\cap\overline J=\emptyset$ and let $\delta>0$.
Then, there exists a constant $C=C(I,J,\delta)>0$ such that for any $f\in C^\infty_c(I)$
\begin{align*}
\|f\|_{L^2(I)}
&\leq C e^{C\frac{\|f\|_{H^1(I)}}{\| f\|_{L^2(I)}}}\|H_\delta f\|_{H^1(J)}.
\end{align*}
\end{prop}

\begin{proof}
As in the proof of Proposition ~\ref{prop:qucTHilbert}, we consider the function $F$ given by \eqref{eq:Ff} and we derive the desired result from a suitable logarithmic stability estimate for $\tilde H_\delta$, which in contrast to the one in Lemma ~\ref{lem:qucTtildeH}, involves $H^1$ measurements of $H_{\delta}$.

Let $\tilde I, \tilde J$ be as in Lemma ~\ref{lem:qucTtildeH}. We claim that for $F\in C^\infty_c(\tilde I)$
\begin{align}\label{eq:THH1F}
\|F\|_{L^2(\tilde I)}
&\leq C e^{C \frac{\|F'\|_{L^2(\tilde I)}}{\|F\|_{L^2(\tilde I)}}}\|\tilde H_\delta F \|_{H^1(\tilde J)},
\end{align}
where $C$ depends on $\tilde I $ and $\tilde J$.
In order to prove this bound, we first rewrite \eqref{eq:tHdexp} as
\begin{align*}
\tilde H_\delta F(\tx)=\sum_{j=0}^\infty \tilde F_j x^{-j},
\end{align*}
where 
\begin{align*}
\tilde F_0 &= \delta F_0,\quad 
\tilde F_j = 2\delta F_j +\frac{1}{\pi} F_{j-1}, \ j\geq 1. 
\end{align*}

Let us assume for the moment that $\tilde I\Subset (0,1)$ and $\tilde J\Subset (1,\infty)$. Notice that by \eqref{eq:fest} and Lemma ~\ref{lem:moment_low}, we have respectively
\begin{align}\label{eq:THH1est1}
\|F\|_{L^2(\tilde I)}^2&\leq C_0 e^{C(N_F+1)}\sum_{j=0}^{N_F} |F_j|^2,
\\
\label{eq:THH1est2}
\sum_{j=0}^{N_F} |\tilde F_j|^2&\leq  e^{C(N_F+1)}\Big(\Big\|\sum_{j=0}^\infty \tilde F_j z^j\Big\|_{L^2(\tilde I)}^2+\Big\|\frac{d}{dz}\sum_{j=0}^\infty \tilde F_j z^j\Big\|_{L^2(\tilde I)}^2\Big),
\end{align}
where $N_F+1\simeq \frac{\|F'\|_{L^2(\tilde I)}}{\|F\|_{L^2(\tilde I)}}$ and $C$ depends on $|\tilde I|$.

On the one hand, we then seek to estimate $\sum_{j=0}^{N_F} |F_j|^2$ in terms of  $\sum_{j=0}^{N_F} |\tilde F_j|^2$.
To this end, we observe that
\begin{align*}
F_j=\frac{1}{2\delta}\sum_{i=0}^j c_{ij} \tilde F_i, 
\end{align*}
with
\begin{align*}
c_{0j}=2\left(\frac{-1}{2\pi\delta}\right)^j, \qquad 
c_{ij}=\left(\frac{-1}{2\pi\delta}\right)^{j-i}, \quad 1\leq i\leq j.
\end{align*}
Therefore,
\begin{align*}
\sum_{j=0}^{N} |F_j|^2
&\leq \frac{N+1}{(2\delta)^2}
\sum_{i=0}^N\sum_{j=i}^N|c_{ij}|^2|\tilde F_i|^2.
\end{align*}
If $\frac{1}{2\pi\delta}\neq1$, we thus obtain
\begin{align*}
\sum_{j=0}^N |c_{0j}|^2&=4\sum _{j=0}^N \left(\frac{1}{2\pi\delta}\right)^{2j}
=4\frac{1-\left(\frac{1}{2\pi\delta}\right)^{2(N+1)}}{1-\left(\frac{1}{2\pi\delta}\right)^2},\\
\sum_{j=i}^N |c_{ij}|^2&=
\sum _{k=0}^{N-i} \left(\frac{1}{2\pi\delta}\right)^{2k}
=\frac{1-\left(\frac{1}{2\pi\delta}\right)^{2(N+1-i)}}{1-\left(\frac{1}{2\pi\delta}\right)^2},\quad i\geq 1.
\end{align*}
Otherwise, i.e. if $\frac{1}{2\pi\delta}=1$, we simply have $|c_{0j}|=2$ and $|c_{ij}|=1$ for $1\leq i\leq j$. Thus, we can estimate the sum by 
\begin{align*}
\sum_{j=i}^N|c_{ij}|^2\leq
\begin{cases}
\frac{4}{1-\left(\frac{1}{2\pi\delta}\right)^2} & \mbox{ if } \frac{1}{2\pi\delta}< 1,\\
2(N+1) & \mbox{ if } \frac{1}{2\pi\delta}= 1,\\
\frac{4 \left(\frac{1}{2\pi\delta}\right)^{2(N+1)}}{\left(\frac{1}{2\pi\delta}\right)^2-1} & \mbox{ if } \frac{1}{2\pi\delta}> 1.\\
\end{cases}
\end{align*}
Using also that $N+1\leq e^{N+1}$, we  infer
\begin{align}\label{eq:THH1est3}
\sum_{j=0}^{N_F} |F_j|^2
&\leq C e^{C(N_F+1)}
\sum_{i=0}^{N_F}|\tilde F_i|^2,
\end{align}
with $C$ depending only on $\delta$.

On the other hand, we seek to express the right hand side of \eqref{eq:THH1est2} in terms of $H_{\delta}$. To this end, we note that
\begin{align*}
\Big\|\sum_{j=0}^\infty \tilde F_j z^j\Big\|_{L^2(\tilde I)}
&=\Big\|\sum_{j=0}^\infty \tilde F_j \tx^{-j-1}\Big\|_{L^2(\tilde I^{-1})}\leq \Big\|\sum_{j=0}^\infty \tilde F_j \tx^{-j}\Big\|_{L^2(\tilde I^{-1})}=\|\tilde H_\delta F\|_{L^2(\tilde I^{-1})},\\
\Big\|\frac{d}{dz}\sum_{j=0}^\infty \tilde F_j z^j\Big\|_{L^2(\tilde I)}
&=\Big\|\tx\frac{d}{d\tx} \sum_{j=0}^\infty \tilde F_j \tx^{-j}\Big\|_{L^2(\tilde I^{-1})}
\leq  \frac{1}{\inf \tilde I}\Big\|\frac{d}{d\tx}\sum_{j=0}^\infty \tilde F_j \tx^{-j}\Big\|_{L^2(\tilde I^{-1})}\leq C\|(\tilde H_\delta F)'\|_{L^2(\tilde I^{-1})}.
\end{align*}
Together with \eqref{eq:THH1est1}, \eqref{eq:THH1est2} and  \eqref{eq:THH1est3}, we then conclude
\begin{align}\label{eq:THH1est4}
\|F\|_{L^2(\tilde I)}\leq C e^{C(N_F+1)}\left(\|\tilde H_\delta F\|_{L^2(\tilde I^{-1})}+\|(\tilde H_\delta F)'\|_{L^2(\tilde I^{-1})} \right).
\end{align}

In order to conclude the proof of \eqref{eq:THH1F}, we propagate the smallness from $\tilde J$ to $\tilde I^{-1}$ by applying Lemma ~\ref{lem:analytcont} to $\tilde H_\delta F$  and to $(\tilde H_\delta F)'$.  Let $\Omega=\conv(\tilde I^{-1}\cup J)$ and $\omega=J$.
For any $k\in \N_0$ and $x\in\Omega$ by the bounds in \eqref{eq:tHFM} for $\tilde H_\delta F$ we estimate
\begin{align*}
\Big|\frac{d^k}{d\tx^k}(\tilde H_\delta F)'(\tx)\Big|
\leq (k+1)!\rho^{-k}\left(\frac{\pi^{-1}+2\delta}{\rho^2}\right)\|F\|_{L^2(\tilde I)}
\leq k!\left(\frac{\rho}{e}\right)^{-k}\left(\frac{\pi^{-1}+2\delta}{\rho^2}\right)\|F\|_{L^2(\tilde I)},
\end{align*}
where $\rho=\dist(\tilde I, \Omega)>0$. Thus, \eqref{eq:analytcontcond} holds for $\tilde H_\delta F$  and  $(\tilde H_\delta F)'$ with $M=\left(\frac{\pi^{-1}+2\delta}{\min\{\rho,\rho^2\}}+\delta \right)\|F\|_{L^2(\tilde I)}$  and $\rho$ replaced by $\frac{\rho}{e}$. Thus, by Lemma ~\ref{lem:analytcont} and  \eqref{eq:THH1est4} 
\begin{align*}
\|F\|_{L^2(\tilde I)}\leq Ce^{C(N_F+1)}\left(\frac{\pi^{-1}+2\delta}{\min\{\rho,\rho^2\}}+\delta \right)^{1-\theta}\|F\|_{L^2(\tilde I)}^{1-\theta}\|\tilde H_\delta F\|_{H^1(\tilde J)}^\theta,
\end{align*}
where $\theta\in(0,1)$ and $C>0$ depend on $\tilde I$ and $\tilde J$, and  hence \eqref{eq:THH1F} follows. 
Applying the same rescaling as in the proof of Lemma ~\ref{lem:qucTtildeH} and renaming $\delta$, \eqref{eq:THH1F} is obtained for any $\tilde I$ and $\tilde J$ as in Lemma ~\ref{lem:qucTtildeH}.

Finally, let us assume without loss of generality that $I, J\subset \R_+$ with $J$ to the right of $I$, let $f\in C^\infty_c(I)$ and let $F$ be given by \eqref{eq:Ff}.
In addition to \eqref{eq:estimatesFf}, we have
\begin{align*}
\|(\tilde H_\delta F)'\|_{L^2(\tilde J)}&=\|(H_\delta f)'(\cdot)e^{-\pi\delta x}\|_{L^2(J)}\leq  \|(H_\delta f )'\|_{L^2(J)}.
\end{align*}
Inserting these estimates to \eqref{eq:THH1F}, the desired result follows.
\end{proof}

\begin{rmk}
We remark that if $f$ has constant sign, we have $\|(H_\delta f )'\|_{L^2(J)}\leq C \|(H_\delta f )\|_{L^2(J)}$, with $C$ depending on $I$, $J$ and $\delta$, so we recover a particular case of \textit{(ii)} in Proposition ~\ref{prop:qucTHilbert} from the estimate in Proposition ~\ref{prop:qucTHilbert1}. Indeed, we observe that for $x\in J$
\begin{align*}
|(H_\delta f )'(x)|&=\delta^2\pi \int_I\csch^2 \big(\pi\delta(x-y)\big) |f(y)|dy
=4\delta^2\pi \int_I \frac{e^{2\pi\delta (x+y)}}{(e^{2\pi\delta x}-e^{2\pi\delta y})^2}  |f(y)|dy\\
&= 4\delta \pi \sup_{y\in I}\left(\frac{e^{2\pi\delta (x+y)}}{e^{4\pi\delta x}-e^{4\pi\delta y}} \right)|H_\delta f(x)|\leq 4\delta \pi \frac{e^{2\pi\delta x} e^{2\pi\delta \sup I}}{e^{4\pi\delta x}-e^{4\pi\delta \sup I}} |H_\delta f(x)|.
\end{align*}
\end{rmk}

\begin{rmk}
\label{rmk:relax_suppII}
As in Remarks \ref{rmk:relax} and \ref{rmk:relax_support_HT} it would have been possible to relax the conditions of $f \in C_c^{\infty}(I)$ or $f\in H^1_0(I)$ to the condition that $f\in H^1(I)$ in the formulation of Propositions \ref{prop:qucTHilbert} and \ref{prop:qucTHilbert1}.
\end{rmk}

\subsection{The inverse fractional Laplacian}
In this section we consider the inverse fractional Laplacian for any  $\alpha\in\R_+\backslash\N_0$, which is given by 
\begin{align*}
-(-\Delta)^{-\alpha} f(x)=-c_\alpha\mathcal F^{-1}\big(|\cdot|^{-2\alpha}\mathcal F f(\cdot)\big)(x)=\int_\R \frac{f(y)}{|x-y|^{1-2\alpha}}dy.
\end{align*}
This corresponds to the Riesz potential $I_{2\alpha}$ for which we prove new quantitative stability estimates. We refer to the work of Riesz \cite{R38} for early qualitative results related to this. Similarly as in Proposition ~\ref{prop:qucTHilbert1}, we can always prove an unconditional estimate if we measure the data with respect to the $H^1(J)$ topology. As in Proposition ~\ref{prop:qucTHilbert}, under certain additional assumptions we may relax this to stability estimates with $L^2(J)$ measurements of the data.

\begin{prop}\label{prop:qucInvFrLapl}
Let $I$ and $J$ be two  open, bounded, non-empty subsets of $\R$ such that $\overline I\cap\overline J=\emptyset$ and let $\alpha\in\R_+$ such that $2\alpha\notin\N_0$. Then there exists $C=C(I,J,\alpha)>0$ such that for any $f\in C^\infty_c(I)$ 
\begin{align}\label{eq:qucInvFrLaplH1}
\|f\|_{L^2(I)}
&\leq C e^{C\frac{\| f'\|_{L^2(I)}}{\| f\|_{L^2(I)}}}\|(-\Delta)^{-\alpha}f\|_{H^1(J)}.
\end{align}
Suppose that in addition one of the following conditions holds:
\begin{enumerate}[label=(\roman*)]
\item $\alpha\in(0,\frac{1}{2})$,
\item $\int_I f(x)x^jdx=0$ for all $j\leq 2m-1$ if $\alpha\in(m,m+\frac{1}{2})$ and $j\leq 2m$ if $\alpha\in(m+\frac{1}{2},m+1)$, where in both cases $m \in \N_0$,
\item $f$ has constant sign.
\end{enumerate}
Then, the $H^1$-norm in \eqref{eq:qucInvFrLaplH1} can be replaced by the $L^2$-norm, i.e.,
\begin{align}\label{eq:qucInvFrLaplL2}
\|f\|_{L^2(I)}
&\leq C e^{C\frac{\| f'\|_{L^2(I)}}{\| f\|_{L^2(I)}}}\|(-\Delta)^{-\alpha}f\|_{L^2(J)}.
\end{align}
\end{prop}

\begin{proof}
Let us assume that $f\in C^\infty_c(I)$ with $I\Subset(0,1)$. For
$x>1$, we have
\begin{align}\label{eq:invfracexp}
\begin{split}
-(-\Delta)^{-\alpha} f(x)&=\int_\R \frac{f(y)}{|x-y|^{1-2\alpha}}dy
=\int_\R \frac{1}{x^{1-2\alpha}}\frac{f(y)}{(1-\frac{y}{x})^{1-2\alpha}}dy
\\&
=\frac{1}{x^{1-2\alpha}}\int_\R f(y)\left(1+\sum_{j=1}^\infty \frac{(1-2\alpha)\dots(j-2\alpha)}{j!}\left(\frac{y}{x}\right)^j\right)dy
\\&
=\sum_{j=0}^\infty c_{j}f_j x^{-j-1+2\alpha},
\end{split}
\end{align}
with 
\[c_0=1, \qquad c_j=\prod_{k=1}^{j}\left(1-\frac{2\alpha}{k}\right), \quad j\geq 1.\]
Notice that if $\alpha\in (0,\frac 1 2)$,  $\{c_j\}_{j\in\N_0}$ is a decreasing sequence of positive numbers with $c_{j}\geq (1-2\alpha)^{j}$. 
Otherwise the signs of the coefficients $c_{j}$ oscillate up to a finite value of $j$ depending on $\alpha$, after which their absolute values start to decrease.
Moreover,
\begin{align}
\label{eq:cj_abs}
|c_j|\geq \left(1-\frac{2\alpha}{k_\alpha}\right)^j \mbox{ for some } k_\alpha\in\N.
\end{align}

Independently of the signs of the coefficients $c_j$, we can apply Lemma ~\ref{lem:moments} with $\gamma_j= c_j$ and obtain an estimate of the form \eqref{eq:mom2}, 
which using \eqref{eq:cj_abs} for the coefficients $|c_j|$ turns into
\begin{align}\label{eq:invfrac1}
\|f\|_{L^2(I)}
\leq \frac{C_0e^{C(N_f+1)}}{\left(1-\frac{2\alpha}{k_\alpha}\right)^{N_f}}\left(\Big\|\sum_{j=0}^\infty c_j f_j z^j\Big\|_{L^2(I)}+\Big\|\frac{d}{dz}\big(\sum_{j=0}^\infty c_j f_j z^j\big)\Big\|_{L^2(I)}\right),
\end{align}
with  $N_f+1\simeq \frac{\|f'\|_{L^2(I)}}{\| f\|_{L^2(I)}}$.
By the change of variables $z=x^{-1}$, we obtain
\begin{align*}
\Big\|\sum_{j=0}^\infty c_{j}f_j z^{j}\Big\|_{L^2(I)}
&\leq \Big\|\sum_{j=0}^\infty c_{j}f_j z^{j-2\alpha}\Big\|_{L^2(I)}
\\&\leq  \Big\|\sum_{j=0}^\infty c_{j}f_jx^{-j-1+2\alpha}\Big\|_{L^2(I^{-1})}
=\|(-\Delta)^{-\alpha} f\|_{L^2(I^{-1})},
\\
\Big\|\frac{d}{dz}\big(\sum_{j=0}^\infty c_j f_j z^j\big)\Big\|_{L^2(I)}
&\leq \Big\|x\frac{d}{dx}\big(x^{1-2\alpha}\sum_{j=0}^{\infty}c_jf_jx^{-j-1+2\alpha}\big)\Big\|_{L^2(I^{-1})}
\\&\leq |1-2\alpha|\sup_I x^{1-2\alpha} \Big\|\sum_{j=0}^\infty c_{j}f_j x^{j-2\alpha}\Big\|_{L^2(I^{-1})}\\
& \qquad+\sup_I x^{2-2\alpha} \Big\|\frac{d}{dx}\big(\sum_{j=0}^\infty c_{j}f_j x^{j-2\alpha}\big)\Big\|_{L^2(I^{-1})}
\\
&\leq C\left( \|(-\Delta)^{-\alpha} f\|_{L^2(I^{-1})}
+ \|\big((-\Delta)^{-\alpha} f\big)'\|_{L^2(I^{-1})}\right)
\end{align*}
with $ I^{-1}\Subset(1,\infty)$. Thus, \eqref{eq:invfrac1} implies
\begin{align}\label{eq:invfrac2}
\|f\|_{L^2(I)}
\leq \frac{Ce^{C(N_f+1) }}{\left(1-\frac{2\alpha}{k_\alpha}\right)^{N_f}}\left(\|(-\Delta)^{-\alpha} f\|_{L^2(I^{-1})}+\|\big((-\Delta)^{-\alpha} f\big)'\|_{L^2(I^{-1})}\right),
\end{align}
with $C$ depending on $I$ and $\alpha$.

In order to control $\|(-\Delta)^{-\alpha} f\|_{H^1(I^{-1})}$ by $\|(-\Delta)^{-\alpha} f\|_{H^1(J)}$, we only need to study the case when $I^{-1}\nsubseteq J$, since otherwise the estimate follows directly. 
We apply Lemma ~\ref{lem:analytcont} with $\Omega=\text{conv}(I^{-1}\cup J)\subset(1,\infty)$, where $(-\Delta)^{-\alpha} f$ is analytic, and $\omega=J$. Notice that for any $k\in\N_0$  and $x\in\Omega$ we have 
\begin{align*}
\left|\frac{d^k}{dx^k} \big((-\Delta)^{-\alpha}f\big)(x)\right|&=|(1-2\alpha)\dots(k-2\alpha)|\left|\int_I\frac{f(y)}{|x-y|^{k+1-2\alpha}}dy\right|
\leq C_\alpha k!\rho^{-k-1+2\alpha}\|f\|_{L^2(I)},\\
\left|\frac{d^k}{dx^k} \big((-\Delta)^{-\alpha}f\big)'(x)\right|&
\leq C_\alpha (k+1)!\rho^{-k-2+2\alpha}\|f\|_{L^2(I)}
\leq C_\alpha k! \left(\frac{\rho}{e}\right)^{-k}\rho^{-2+2\alpha}\|f\|_{L^2(I)},
\end{align*}
where $\rho=\dist(I, \Omega)>0$.
Then \eqref{eq:analytcontcond} holds  for both $(-\Delta)^{-\alpha}f$ and its derivative with $M=C_\alpha \frac{\rho^{2\alpha}}{\min\{\rho, \rho^2\}}\|f\|_{L^2(I)}$ and $\rho$ replaced by $\frac{\rho}{e}$.
By Lemma ~\ref{lem:analytcont}, \eqref{eq:invfrac2} turns into 
\begin{align*}
\|f\|_{L^2(I)}
\leq \frac{Ce^{C(N_f+1)}}{\left(1-\frac{2\alpha}{k_\alpha}\right)^{N_f}}\|f\|_{L^2(I)}^{1-\theta}\left(\|(-\Delta)^{-\alpha}f\|_{L^2(J)}^\theta+\|\big((-\Delta)^{-\alpha}f\big)'\|_{L^2(J)}^\theta\right),
\end{align*} 
where $\theta\in (0,1)$ and $C>0$ depend on $I, J$ and $\alpha$.
Therefore,  \eqref{eq:qucInvFrLaplH1} follows.

Let us now consider the particular settings \textit{(i)-(iii)} in which it is possible to improve the $H^1(J)$ to an $L^2(J)$-bound in \eqref{eq:qucInvFrLaplH1}:

\begin{enumerate}[label=\textit{(\roman*)}]
\item As explained above, if $\alpha\in (0,\frac 1 2)$, the sequence $\{c_j\}_{j\in\N_0}$ is a decreasing sequence of positive numbers.
Hence, we can apply  the estimate \eqref{eq:mom1} in Lemma ~\ref{lem:moments} with $\gamma_j= c_j$ to directly obtain that
\begin{align*}
\|f\|_{L^2(I)}
\leq \frac{C_0e^{C(N_f+1)}}{(1-2\alpha)^{N_f}}\Big\|\sum_{j=0}^\infty c_j f_j z^j\Big\|_{L^2(I)}.
\end{align*}
Proceeding as before to estimate the last norm in terms of $\|(-\D)^{-\alpha} f\|_{L^2(J)}$,  \eqref{eq:qucInvFrLaplL2} is obtained.

\item
If $\alpha>\frac{1}{2}$   the signs of the coefficients $c_{j}$ oscillate according to 
\begin{align*}
\{\sign(c_{j})\}_{j\in\N_0}=\{+,(-+)^m,+,+,\dots\},
&\mbox{ if }\alpha\in(m,m+\frac{1}{2}), \\
\{\sign(c_{j})\}_{j\in\N_0}=\{+,(-+)^m,-,-,\dots\},
&\mbox{ if }\alpha\in(m+\frac{1}{2},m+1),
\end{align*}
for any $m\in \N_0$.
If we consider  $\gamma_j=|c_j|$, we can  apply  the estimate \eqref{eq:mom1} in Lemma ~\ref{lem:moments}, yielding
\begin{align*}
\|f\|_{L^2(I)}
&\leq \frac{C_0e^{C(N_f+1)}}{(1-\frac{2\alpha}{k_\alpha})^{N_f}}\Big\|\sum_{j=0}^\infty |c_j| f_j z^j\Big\|_{L^2(I)}
\\ &\leq \frac{C_0e^{C(N_f+1)}}{(1-\frac{2\alpha}{k_\alpha})^{N_f}}
 \left(\Big\|\sum_{j=0}^\infty c_j f_j z^j\Big\|_{L^2(I)}+2\Big\|\sum_{j\in J_\alpha}c_j f_j z^j\Big\|_{L^2(I)}\right),
\end{align*}
where $J_\alpha=\{1,3,\dots,2m-1\}$ for $\alpha\in(m,m+\frac{1}{2})$ and  $J_\alpha=\{0,2,\dots,2m\}$ for $\alpha\in(m+\frac{1}{2},m+1)$.
The last term vanishes if the moments $\{f_j\}_{j\in J_\alpha}$ are zero and  \eqref{eq:qucInvFrLaplL2} can be  inferred repeating the previous estimates.

Up to now we have assumed that $I\subset (0,1)$. In general, the claim in \textit{(ii)} follows by a reduction to the setting in which $I \subset (0,1)$ by scaling and translation. Indeed, 
let $p, q$ be such that $\{y=px+q: x\in I\}=(0,1)$ and ${F}(y)=f(x)\in C^\infty_c(0,1)$. Therefore the moments of {$F$} are combinations of the integrals $\int_If(x)x^jdx$ according to
\begin{align*}
F_j=\int_0^1  F(y) y^jdy=p\int_I f(x)(pm+q)^j dx =
\sum_{k=0}^j {{j}\choose{k}}p^{k+1}q^{j-k} \int_If(x)x^kdx.
\end{align*}
If condition \textit{(ii)} is satisfied, it is clear that then $F_j=0$ for $j\in J_\alpha$. 

\item
Let us assume that $f$ has a constant sign.
Then for $x\in J$,
 \begin{align*}
|((-\Delta)^{-\alpha} f)'(x)|=|1-2\alpha|\int_I \frac{|f(y)|}{|x-y|^{1-2\alpha+1}}dy\leq \frac{|1-2\alpha|}{\dist(I, J)^{2-2\alpha} }|(-\Delta)^{-\alpha} f(x)|.
\end{align*}
Thus,  
\begin{align*}
\|(-\D)^{-\alpha} f\|_{H^1(J)}\leq C \|(-\D)^{-\alpha} f\|_{L^2(J)},
\end{align*}
where $C$ depends on $I, J$ and $\alpha$.
In this case \eqref{eq:qucInvFrLaplH1} reduces to \eqref{eq:qucInvFrLaplL2}.
\end{enumerate}
\end{proof}

\section{Quantitative Unique Continuation through Stability of the Moment Problem, Higher Dimensions}

\label{sec:moment_gen}

The previous ideas based on the stability of the moment problem can be also applied to prove quantitative unique continuation results for other operators like the Fourier transform or the Laplace transform, also in higher dimensions.

In our discussion of the higher dimensional problem, in this section we restrict ourselves to a class of operators which generalize the Fourier and the Laplace transforms. 
To this end, let $\alpha,\beta\in\R$, $(\alpha, \beta)\neq(0,0)$. 
We define the operator $T_{\alpha,\beta}$ acting on  compactly supported  functions in $\R^n$ as
\begin{align*}
\mathcal T_{\alpha,\beta} f(x)=\int_{\R^n} e^{(\alpha+i\beta) x\cdot y }f(y)dy.
\end{align*}
Notice that the Laplace and the Fourier transforms are particular cases of $\mathcal T_{\alpha, \beta}$:
\begin{itemize}
\item $T_{-1,0} f=\mathcal L f$,
\item $T_{0,-1} f=\mathcal F f$.
\end{itemize}

In addition to exploiting the previous ideas on the stability of the moment problem and the quantitative unique continuation for real analytic functions, our proof of the stability result for these operators also relies on the fact that $\mathcal T_{\alpha,\beta} f$ has an analytic extension to $\C^n$ by the Paley-Wiener-Schwartz theorem for compactly supported functions $f$. 

The logarithmic stability result for the operators $\mathcal T_{\alpha,\beta}$ is analogous to Theorem ~\ref{thm:qucHilbert_a} although the disjointness of the subsets $I$ and $J$ is not necessary since the kernel is entire.

\begin{prop}\label{prop:qucexphigher}
Let $ I$ and $J$ be two  open, bounded, non-empty subsets of $\R^n$.  Then for any $\alpha, \beta \in \R$ with $(\alpha, \beta) \neq (0,0)$ there exists $C=C(I, J,n,\alpha,\beta)>0$ such that for any $f\in C^\infty_c( I)$ 
\begin{align*}
\|f\|_{L^2( I)}
&\leq C e^{C\frac{\|\nabla f\|_{L^2( I)}} {\| f\|_{L^2( I)}}\left|\log \left(\frac{\|\nabla f\|_{L^2( I)}} {\| f\|_{L^2( I)}}\right)\right|} 
\|\mathcal T_{\alpha,\beta} f\|_{L^2( J)}.
\end{align*} 
\end{prop}

Due to their relevance in harmonic analysis and inverse problems, the discussion of stability properties for the (truncated) Fourier and Laplace transforms is not new: In \cite[Lemma 4.1]{CFDSR14} the authors rely on analytic continuation arguments to prove logarithmic stability of the Fourier transform, $\mathcal{T}_{0,-1}f$, in any dimension and apply it to deduce logarithmic stability properties for the Radon transform. In \cite{LS17} the authors derive optimal logarithmic stability for the one-dimensional truncated Laplace ($\alpha =-1, \beta=0$) and Fourier transforms ($\alpha = 0, \beta=-1$) by means of Slepian's miracle \cite{SP61}. We also refer to \cite{Kov01} for related quantitative results on the Logvinenko-Sereda theorem.

In the present work we discuss the Fourier-Laplace transform primarily as an example of how our methods apply to the higher dimensional setting and thus provide a unified treatment of the Fourier-Laplace transform in arbitrary dimension and without necessarily requiring an openness condition for $J$ (which can be seen from the proof below, see also Remark \ref{rmk:Jmeasure}).

\begin{proof}
Without loss of generality, let us assume that $ I\subset (0,1)^n$. 
Let $x\in\R^n$, then we can expand $\mathcal T_{\alpha,\beta} f(x)$ as follows
\begin{align*}
\mathcal T_{\alpha,\beta} f(x)
&=\int_{(0,1)^n} e^{(\alpha+i\beta) x\cdot y }f(y)dy
=\int_{(0,1)^n} f(y)\sum_{k=0}^\infty \frac{\big((\alpha+i\beta) x\cdot y\big)^k}{k!}dy
\\&=\int_{(0,1)^n} f(y)\sum_{j\in\N_0^n} \frac{\big((\alpha+i\beta)x\big)^j y^j}{j!}dy
=\sum_{j\in\N_0^n} \frac{1}{j!}f_j\big((\alpha+i\beta)x\big)^j,
\end{align*}
where $j!=j_1!\dots j_n!$.
Notice that both the definition of $\mathcal T_{\alpha,\beta} f$ and the expansion can be extended to any $z\in\C^n$.

By Lemma ~\ref{lem:moments} with $\gamma_j=\frac{1}{j!}$ we have
\begin{align}\label{eq:Tf1}
\|f\|_{L^2( I)}
\leq \frac{C_0e^{C(N_f+1)}}{\frac{1}{(N_f!)^2}}\Big\|\sum_{j\in\N_0^n} \frac{1}{j!}f_j x^j\Big\|_{L^2( I)}
\end{align}
with  $N_f+1\simeq \frac{\|\nabla f\|_{L^2(I)}}{\| f\|_{L^2(I)}}$.
By the change of variables $x=(\alpha+i\beta)z$ we infer
\begin{align}\label{eq:Tf2}
\Big\|\sum_{j\in\N_0^n} \frac{1}{j!}f_j x^j\Big\|_{L^2( I)}
\leq \sup_{x\in I} \Big|\sum_{j\in\N_0^n} \frac{1}{j!}f_j x^j\Big|
\leq \sup_{z\in I_{\alpha,\beta}} \Big|\sum_{j\in\N_0^n} \frac{1}{j!}f_j \big((\alpha+i\beta)z\big)^j\Big|
\leq \sup_{z\in I_{\alpha,\beta}}|\mathcal T_{\alpha, \beta}f(z)|,
\end{align}
where
\[ I_{\alpha,\beta}=\left\{\frac{\alpha-i\beta}{|\alpha+i\beta|^2}x:x\in I\right\}.\]
Let $\ell=|\alpha+i\beta|^{-1}$ and $r=\sqrt{n}\ell$. Notice that if $\beta=0$, then $I_{\alpha,0}\subset (-\ell,\ell)^n\subset \R^n$, whereas if $\beta \neq 0$, $\beta\in\R_\pm$, then  $I_{\alpha,\beta}\subset B^\mp_{r} \subset (\C_\mp)^n$.

In order to translate the smallness of $\mathcal T_{\alpha,\beta} f$ in $ J$ into $ I_{\alpha,\beta}$, firstly we make use of  the real analyticity of $\mathcal T_{\alpha,\beta} f$ by means of Lemma ~\ref{lem:analytcont} with $\Omega=\text{conv }((-2r,2r)^n\cup J)$ and $\omega=J$. For any  $k\in\N_0^n$ and $x\in\Omega$ we have that 
\begin{align*}
\left|\partial^k_x\mathcal T_{\alpha,\beta} f(x)\right|=\left|(\alpha+i\beta)^{|k|}\int_{(0,1)^n} f(y) y^ke^{(\alpha+i\beta)x\cdot y}dy\right|\leq C_M \rho^{-|k|}\|f\|_{L^2( I)}, 
\end{align*}
with $\rho=|\alpha+i\beta|^{-1}=\ell$ and $C_M$ depending on $\Omega, \alpha$ and $\beta$. Then, \eqref{eq:analytcontcond}  holds with $M=C_M\|f\|_{L^2(I)}$ and therefore there exist $C>0$ and $\theta_1\in(0,1)$ depending on $I, J$ and the parameters $\alpha,\beta$ such that
\begin{align}\label{eq:Tf3}
\sup_{(-2r,2r)^n} |\mathcal T_{\alpha,\beta} f|\leq C\|f\|_{L^2( I)}^{1-\theta_1}\|\mathcal T_{\alpha,\beta}f\|_{L^2( J)}^{\theta_1}.
\end{align}
If $\beta=0$, since $I_{\alpha,0}\subset (-2r,2r)^n$, the proof can be concluded as follows: We already have by \eqref{eq:Tf1}, \eqref{eq:Tf2} and \eqref{eq:Tf3} that 
\begin{align*}
\|f\|_{L^2( I)}
&\leq  C (e^{c(N_f+1)}N_f!)^{\frac{1}{\theta_1}} \|\mathcal T_{\alpha,\beta}f\|_{L^2( J)}
\leq  C (e^{c(N_f+1)\log (N_f+1)})^{\frac{1}{\theta_1}} \|\mathcal T_{\alpha,\beta}f\|_{L^2( J)},
\end{align*}
where in the last step we have used that $N!\leq N^N$ for $N\geq 1$.

If $\beta\neq 0$ we use the complex analyticity of $\mathcal T_{\alpha,\beta}f$. Without loss of generality, we may assume that $\beta<0$. By  Corollary ~\ref{cor:analytcontJohn} there is $\theta_2\in(0,1)$ such that
\begin{align*}
\sup_{ B^+_{r}} |\mathcal T_{\alpha,\beta} f|
&\leq  \big(\sup_{(-2r,2r)^n} |\mathcal T_{\alpha,\beta} f|\big)^{\theta_2}
\big(\sup_{ B^+_{4n\ell}} |\mathcal T_{\alpha,\beta}f|\big)^{1-\theta_2},
\end{align*}
where $r=\sqrt{n}\ell$. 
For any $z\in B^+_{4n\ell}$ we have
\begin{align*}
|\mathcal T_{\alpha,\beta} f(z)|=\left|\int_{(0,1)^n} e^{(\alpha+i\beta)z\cdot y}f(y)dy\right|\leq e^{4n^2\ell|\alpha+i\beta|}\|f\|_{L^2(I)}\leq e^{4n^2}\|f\|_{L^2(I)}.
\end{align*}
Combining this with the bound \eqref{eq:Tf3} leads to
\begin{align}\label{eq:Tf4}
\sup_{B_r^+} |\mathcal T_{\alpha,\beta} f|
& \leq C \|f\|_{L^2( I)}^{1-\theta}\|\mathcal T_{\alpha,\beta}f\|_{L^2( J)}^{ \theta},
\end{align}
with $\theta=\theta_1\theta_2$.
Thus, we have translated the smallness of $T_{\alpha, \beta}f|_{J}$ to a region in the complex plane which contains $I_{\alpha,\beta}$ (see Figure ~\ref{fig:JIab} for $n=1$).
Finally, from \eqref{eq:Tf1}, \eqref{eq:Tf2} and \eqref{eq:Tf4} we conclude
\begin{align*}
\|f\|_{L^2( I)}
&\leq  C (e^{c(N_f+1)}N_f!)^{\frac{1}{\theta}} \|\mathcal T_{\alpha,\beta}f\|_{L^2( J)},
\end{align*}
which leads to the claimed estimate.

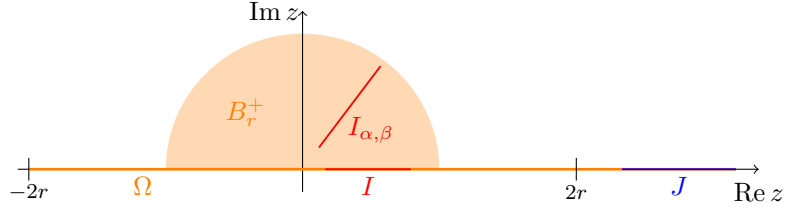
\begin{figure}
\begin{center}
\begin{tikzpicture}[scale=1.5]

\pgfmathsetmacro{\aa}{3/6}
\pgfmathsetmacro{\bb}{-4/6}
\pgfmathsetmacro{\rr}{1/sqrt(\aa*\aa+\bb*\bb)}
\pgfmathsetmacro{\Ja}{2.8}
\pgfmathsetmacro{\Jb}{3.8}
\pgfmathsetmacro{\Ia}{.2}
\pgfmathsetmacro{\Ib}{.95}

\begin{scope}
    \clip (-2*\rr,0) rectangle (2*\rr,\rr);

\draw[orange!0, fill=orange, opacity=.3](0,0) circle [radius=\rr];
\node[orange] at (-.5,.5) {$B_r^+$};
\end{scope}

\draw[->] (-2^\rr-.2,0)--(\Jb+.2,0);
\draw[->] (0,-.2)--(0,\rr+.2);
\node at (-.25,\rr+.2) {$\im z$};
\node at (\Jb+.2,-.2) {$\re z$};

\draw[orange, line width=1pt] (-2*\rr,0)--(\Jb,0);
\node[orange] at (-2*\rr+1,-.15) {$ \Omega$};
\draw (-2*\rr,-.1)--(-2*\rr,.1);
\node at (-2*\rr,-.2) {\footnotesize $ -2r$};
\draw (2*\rr,-.1)--(2*\rr,.1);
\node at (2*\rr,-.2) {\footnotesize $2r$};

\draw[red, line width=.7pt] (\Ia,0)--(\Ib,0);
\node[red] at (\Ia/2+\Ib/2,-.15) {$ I$};

\draw[blue, line width=.7pt] (\Ja,0)--(\Jb,0);
\node[blue] at (\Ja/2+\Jb/2,-.15) {$ J$};

\draw[red, line width=.7pt] (\Ia*\aa*\rr*\rr, -\Ia*\bb*\rr*\rr)--(\Ib*\aa*\rr*\rr, -\Ib*\bb*\rr*\rr);
\node[red] at (.6,.35) {$ I_{\alpha,\beta}$};
\end{tikzpicture}
\end{center}
\caption{Subsets for the propagation of smallness in the proof of Proposition ~\ref{prop:qucexphigher}. By real analyticity, we propagate the information from $J$ to $(-2r, 2r)$. Notice that for $n=1$, $r=\ell=|\alpha+i\beta|^{-1}$. Then, we use complex analyticity to reach $B^+_r$, which contains $I_{\alpha,\beta}$.
}
\label{fig:JIab}
\end{figure}

\end{proof}

\begin{rmk}
We remark that our methods are not limited to the examples given. Another class of operators that could immediately be treated with these ideas in higher dimensions are for instance operators of the form
\begin{align*}
T_\alpha f(x)=\int_{\R^n} \frac{f(y)}{\big(x\cdot(x-y)\big)^\alpha} dy,
\end{align*}
for any $\alpha\in\R\backslash\{0,-1,\dots\}$.
Notice that the kernel can be expanded as
\begin{align*}
\frac{1}{\big(x\cdot(x-y)\big)^\alpha}=\frac{1}{|x|^{2\alpha}}\frac{1}{\big(1-\frac{x\cdot y}{|x|^2}\big)^\alpha}=
\frac{1}{|x|^{2\alpha}}\sum_{k=0}^\infty c_{\alpha,k}\left(\frac{x\cdot y}{|x|^2}\right)^k.
\end{align*}
Therefore, arguing as before we can write $T_\alpha f$ in terms of the moments. Then, with a suitable change of variables and an application of Lemma ~\ref{lem:moments}, independently of the sign of $c_{\alpha,j}$, the quantitative unique continuation result follows (with the data measured in the $H^1(J)$-norm unless $\alpha\in \R_-$).
\end{rmk}

\part{Logarithmic Stability Results Using the Branch-Cut Argument }
\label{part_2}

In the following sections we introduce a second method for proving logarithmic stability estimates which relies on the presence of branch-cut singularities in the symbols of the operators under consideration. 

\section{Auxiliary Results for the Branch-Cut Argument}
\label{sec:aux}

In this section we collect auxiliary results on quantitative unique continuation in one and higher dimensions, on the relation between analyticity and locality and on some boundary-bulk estimates for holomorphic extensions.

\subsection{Quantitative analytic continuation}

\label{sec:ana_ext}

In this section we prepare for the propagation of smallness arguments which we will apply in the following sections. The key result here is Theorem ~\ref{thm:analytic1D} which is a quantitative analytic continuation argument providing the framework for the logarithmic stability estimates obtained in the sequel.

\begin{thm}[Quantitative analytic continuation]
\label{thm:analytic1D}
Let $I, J \subset \R$ be  open, connected, bounded, non-empty intervals with $\overline{I}\cap \overline{J} = \emptyset$. Let $h:\R\to\C$ be  such that $h\in C^0(\R,\C)$,  $\im h|_{J}=0$, $\re h|_J$ is real analytic and $h$ has an analytic extension $\tilde h$ into the upper complex half-plane.
Let $K\subset\R$ be a bounded interval containing $\conv(2I\cup 2J)$ and assume that
 \begin{align}\label{eq:analytic1Dcond}
\|\tilde h\|_{H^{s_1}(K\times[0,2])}+\| h\|_{H^{s_2}(J)}\leq M
\end{align}
for some $s_1>0$, $s_2>1$.
Then for any $s>0$ there exist $\nu>0$ and $C>0$ such that 
\begin{align*}
\|\tilde h\|_{L^2(I\times[0,1])}\leq CM\frac{1}{\left|\log \left(C\frac{ \| h\|_{H^{-s}(J)}}{M}\right) \right|^\nu}.
\end{align*}
\end{thm}

\begin{rmk}
Here and in the following, when writing that $h\in C^{0}(\R, \C)$ has an analytic extension  into the upper complex half-plane, we mean that there exists a function $\tilde{h}: \C_+ \rightarrow \C$ which is analytic as a function in $\{x\in\C: \im(x) >0 \}$ and continuous up to the boundary such that $\tilde{h}|_{\{\im(x)=0\}} = h$. 
We will exploit the existence of analytic continuations into the upper half-plane in the following sections in the construction of our comparison operators.
\end{rmk}

\begin{rmk}\label{rmk:lower}
We observe that  Theorem ~\ref{thm:analytic1D} also holds for functions with an analytic extension into the lower half-plane which directly follows by reflecting the subsets in the upper half-plane into the lower half-plane.
\end{rmk}

In order to prove Theorem ~\ref{thm:analytic1D}, we rely on the following result, which is an analogue of Lemma ~\ref{lem:analytcontJohn} for $L^2$-norms:

\begin{lem}\label{lem:John}
Let $I$, $J$, $h$ and $\tilde h$ be as in Theorem ~\ref{thm:analytic1D}. Then there exist $\beta\in(0,1)$ and $C>0$ such that
\begin{align*}
\|\tilde h\|_{L^2(J/4\times[0,1])}\leq C \| h\|_{H^1(J)} ^\beta \|\tilde h\|_{L^2(2J\times[0,2])}^{1-\beta}.
\end{align*}
\end{lem}

\begin{proof}
We exploit the fact that the real and imaginary parts of  $\tilde{h}$ are harmonic.
If $u$ is harmonic on $2J\times (0,2)$,
by e.g. \cite[Lemma 14.5]{JL99} or \cite[Proposition 5.13]{RS17}, there exists $\beta\in(0,1)$ such that for every $\epsilon\in(0,1)$
\begin{align}\label{eq:harmonic}
\|u\|_{L^2(J/4 \times[\epsilon,1])}\leq C \left(\|u\|_{H^1(J/2 \times \{\epsilon\} )}+\|\partial_y u\|_{L^2(J/2 \times \{\epsilon\})}\right)^\beta\|u\|_{L^2(2J\times[\epsilon,2])}^{1-\beta}.
\end{align}

We now use the Cauchy-Riemann equations $\partial_y \tilde{h}=i\partial_x \tilde{h}$ to dispose of the normal derivative. More precisely, we have 
\begin{align*}
\p_y (\re \tilde h)|_{J/2 \times \{\epsilon\}}=-\p_x (\im \tilde{h})|_{J/2 \times \{\epsilon\}}, \quad 
\p_y(\im\tilde h)|_{J/2 \times \{\epsilon\}}=\p_x (\re \tilde{h})|_{J/2 \times \{\epsilon\}}.
\end{align*}
Applying  \eqref{eq:harmonic} to $\re \tilde{h}$ and $\im \tilde h$ and using the previous identities, we obtain
\begin{align}
\label{eq:appl_CR}
\begin{split}
\|\re \tilde h\|_{L^2(J/4\times[\epsilon,1])}
&\leq C \left(\|\re \tilde{h}\|_{H^1(J/2 \times \{\epsilon\})}+\|\partial_y (\re \tilde h)\|_{L^2(J/2 \times \{\epsilon\})}\right)^\beta\|\re \tilde h\|_{L^2(2J\times[\epsilon,2])}^{1-\beta}\\
&\leq C \left(\|\re \tilde{h}\|_{H^1(J/2 \times \{\epsilon\})}+\|\partial_x (\im \tilde h)\|_{L^2(J/2 \times \{\epsilon\})}\right)^\beta\|\re \tilde h\|_{L^2(2J\times[\epsilon,2])}^{1-\beta},\\
\|\im \tilde h\|_{L^2(J/4\times[\epsilon,1])}
&\leq C \left(\|\im \tilde{h}\|_{H^1(J/2 \times \{\epsilon\})}+\|\partial_y (\im \tilde{h})\|_{L^2(J/2 \times \{\epsilon\})}\right)^\beta\|\im \tilde h\|_{L^2(2J\times[\epsilon,2])}^{1-\beta}\\
&\leq C\left( \|\im \tilde{h}\|_{H^1(J/2 \times \{\epsilon\})}+ \|\p_x (\re \tilde h)\|_{L^2(J/2 \times \{\epsilon\})}\right)^\beta\|\im \tilde h\|_{L^2(2J\times[\epsilon,2])}^{1-\beta}.
\end{split}
\end{align}
Now, using that $\tilde{h}|_J=h|_J$ is itself (real) analytic (and thus in particular $C^1$), (interior in $J \times [0,2]$) regularity theory (up to the boundary $y=0$) for harmonic functions implies that $\tilde{h}$ is ($C^1$) regular up to the boundary $y=0$ in $J/2 \times [0,1]$. It is therefore possible to pass to the limit $\epsilon \rightarrow 0$ in \eqref{eq:appl_CR}. 
Recalling that $\re \tilde{h}|_{J} = \re {h}|_{J} = h|_J$ then concludes the proof.
\end{proof}

\begin{proof}[Proof of Theorem ~\ref{thm:analytic1D}]
We argue in three steps, dealing with an estimate for $\tilde h$ in the interval $I\times [0,\tau]$, an estimate for $I \times [\tau,1]$ and a concatenation step separately.

\medskip

\textit{Step 1: Estimate on $I\times[0,\tau]$.}
As a first step, we observe that by the H\"older and Sobolev inequalities  we have 
\begin{align}\label{eq:f+striptau}
\|\tilde h\|_{L^2(I\times[0,\tau])}\leq \left(|I|\tau\right)^{\frac{1}{q}}\|\tilde h\|_{L^p(I\times[0,\tau])}\leq C\tau^{\frac{1}{q}}\|\tilde h\|_{H^{s_1}(I\times[0,\tau])}
\leq C\tau^{\frac 1 q} M
\end{align}
with $\frac{1}{p}+\frac{1}{q}=\frac{1}{2}$ and $s_1>1-\frac{2}{p}$.\\

\textit{Step 2: Estimate on $I\times[\tau,1]$.}  We seek to transport the smallness of $\|h\|_{H^{-s}(J)}$  to $\|\tilde h\|_{L^2(I\times [\tau,1])}$. 
By Lemma ~\ref{lem:John}
\begin{align*}
\|\tilde h\|_{L^2(J/4\times[0,1])}\leq C\| h\|_{H^1(J)}^{\beta}\|\tilde h\|_{L^2(2J\times[0,2])}^{1-\beta}.
\end{align*}
Since we are interested in  the appearance of $\|h\|_{H^{-s}(J)}$, we note that by interpolation, for any $s>0$ and $s_2>1$, there exists  $\theta=\theta(s,s_2)\in(0,1)$  such that
\begin{align}\label{eq:interpolation}
 \| h\|_{H^1(J)}\leq C \| h\|_{H^{-s}(J)}^{\theta} \| h\|_{H^{s_2}(J)}^{1-\theta}.
\end{align}
Therefore,
\begin{align}\label{eq:appJohn}
\|\tilde h\|_{L^2(J/4\times[0,1])}
\leq C \| h\|_{H^{-s}(J)}^{\alpha_0}M^{1-\alpha_0}.
\end{align}

Having transferred information from the boundary into the interior of the upper half-plane, we next seek to apply a three balls inequality to transfer information from $J/2 \times [0,1]$ to the strip $I\times[\tau,1]$.
This is achieved by an iterated three-balls inequality: More precisely, for a harmonic function $u$, there are $\alpha\in(0,1)$ and $C>0$ such that 
\begin{align}
\label{eq:three_balls}
\|u\|_{L^2(B_{2r}(x))}\leq C \|u\|_{L^2(B_r(x))}^\alpha \|u\|_{L^2(B_{4r}(x))}^{1-\alpha},
\end{align}
for any $B_{4r}(x)$ completely contained in the upper half-plane (see for instance \cite{ARRV09} and the references therein).
We can iterate this estimate along a chain of balls (see Figure ~\ref{fig:2}) contained in $K\times(0,2)$, where $K\supseteq \conv(2I\cup 2J)$. To cover the strip $I\times [\tau,1]$ without touching $I\times\{0\}$, we need to iterate the three balls estimate \eqref{eq:three_balls} roughly $N\sim N_0-C\log \tau$ times, with $N_0$ and $C$ depending just on $I$ and $J$. Therefore,
\begin{align*}
\|u\|_{L^2(I\times [\tau,1])}\leq C \|u\|_{L^2(J/4\times[0,1])}^{\alpha^N} \|u\|_{L^2(K\times[0,2])}^{1-\alpha_0\alpha^N}.
\end{align*}
Applying this estimate to $\re \tilde h$ and $\im\tilde h$  together with  \eqref{eq:appJohn} and \eqref{eq:analytic1Dcond} we obtain 
\begin{align}\label{eq:f+strip1}
\|\tilde h\|_{L^2(I\times [\tau,1])}\leq C \| h\|_{H^{-s}(J)}^{\alpha_0\alpha^N} M^{1-\alpha^N}.
\end{align}

\begin{figure}
\begin{center}
\begin{tikzpicture}[scale=.8]

\draw[->] (-.5,0)--(13.5,0);
\draw[->] (0,-.5)--(0,2);
\node at (-.5,1.8) {$x_{n}$};
\node at (13.5,-.2) {$x'$};

\pgfmathsetmacro{\Ia}{1}
\pgfmathsetmacro{\Ib}{5}
\draw[red] (\Ia,0)--(\Ib,0);
\node[red] at (\Ib/2+\Ia/2,-.2) {$I$};

\pgfmathsetmacro{\Jc}{10}
\pgfmathsetmacro{\Jr}{3}

\fill[blue!20] (\Jc-\Jr/4, 0) rectangle (\Jc+\Jr/4, 1);
\node[blue!50] at (\Jc,1.8) {$\frac{J}{4}\times[0,1]$};

\draw[blue] (\Jc-\Jr,0)--(\Jc+\Jr,0);
\node[blue] at (\Jc,-.2) {$J$};

\pgfmathsetmacro{\t}{.22}
\fill[red!20] (\Ia,\t) rectangle (\Ib,1);

\node[red!50] at (\Ib/2+\Ia/2,1.8) {$I\times[\tau,1]$};

\pgfmathsetmacro{\ra}{\t/2}
\pgfmathsetmacro{\ya}{\t+\t/4}
\pgfmathsetmacro{\sa}{3/4*\ra}
\pgfmathsetmacro{\Ka}{(\Ib-\Ia-\ra)/3*2/\ra)+1}
\foreach \k in {0,...,\Ka} {\draw[line width=.1pt] (\Ia+\k*2*\sa, \ya) circle [radius=\ra];} 

\def\myrec#1#2#3#4#5{
  \ifnum#4<6\relax
  \pgfmathsetmacro{\g}{#1/100}   
    \pgfmathsetmacro{\r}{1/12*(8*#2-4*#1-32*\g+sqrt(16*#2*#2-16* #2*#1-23*#1*#1-128*#2*\g+64*#1*\g+256*\g*\g))}
    \pgfmathsetmacro{\y}{2*\r+\g}
    \pgfmathtruncatemacro{\c}{#4}
    \pgfmathsetmacro{\si}{#5+#1*3/4}
    \pgfmathtruncatemacro{\Kinf}{int(2/3-2/3\si/\r)-1}
    \pgfmathsetmacro{\Ksup}{(\Ib-\Ia-\si-\r)/3*2/\r)+1}
    \ifodd#4 \else
\foreach \k in {\Kinf,...,\Ksup} {\draw[line width=.1pt] (\Ia+\si+\k*3/2*\r, \y) circle [radius=\r];}  
\fi

\pgfmathtruncatemacro{\contador}{#4+1}
    \edef\rep{\noexpand \myrec{\r}{\y}{0}{\contador}{\si};}
    \rep
  \fi
}

\myrec{\ra}{\ya}{0}{1}{0}

\pgfmathsetmacro{\K}{12}
\foreach \k in {-1,...,\K} 
{\draw[line width=.1pt] (\Ia+\si+3/2*\r*\k, \y+\r/2) circle [radius=\r];}
\foreach \k in {1,...,5} 
{\pgfmathsetmacro{\cx}{\Ia+\si+3/2*\r*(\K+\k/3*cos(-10*\k))}
\pgfmathsetmacro{\cy}{\y+\r/2+3/4*\r*\k/3*(sin(-10*\k)}
\draw[line width=.1pt] (\cx, \cy) circle [radius=\r/2];
}

\end{tikzpicture}
\end{center}
\caption{An illustration of the propagation of smallness argument in the proof of Theorem ~\ref{thm:analytic1D}. We first propagate the data from the interval $J$ to the rectangle $\frac{J}{4} \times [0,1]$ by an application of Lemma ~\ref{lem:John}. Next, we iterate the three balls inequality along a chain of balls which reach $I\times[\tau, 1]$. We repeat this process until we cover  $I\times[\tau, 1]$ with the constraint that the balls of double size are not allowed to intersect $I\times\{0\}$. This force us to take balls whose radii are proportional to $x_{n}$. As a consequence, the number of these balls is then proportional to $|\log \tau|$.
}

\label{fig:2}

\end{figure}
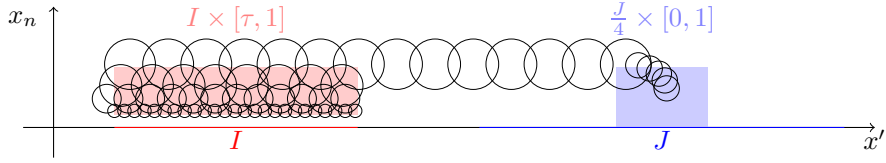

\medskip

\textit{Step 3: Optimization.}
We combine the estimates \eqref{eq:f+striptau} and \eqref{eq:f+strip1} to finally deduce
\begin{align*}
\|\tilde h\|_{L^2(I\times[0,1])}&\leq \|\tilde h\|_{L^2(I\times[0,\tau])}+ \|\tilde h\|_{L^2(I\times[\tau,1])}\\
&\leq C M
\left(\tau^{\frac{1}{q}}+\left(\frac{ \| h\|_{H^{-s}(J)}}{ M}\right)^{\gamma_0\tau^{ \gamma}}\right).
\end{align*}

Optimizing the right hand side in $\tau>0$ (as, for instance, in \cite{ARRV09}), we arrive at the desired estimate.
\end{proof}

Seeking to also deduce results in higher dimensions, we extend Theorem ~\ref{thm:analytic1D} to this setting. Here it will be convenient to select a special direction and to split the variables into this direction and the remaining ones. We hence use the following notation $x=(x', x_n)\in\R^n$.

\begin{thm}\label{thm:analyticnD}
Let $I, J \subset \R$ be  open, bounded, non-empty intervals with $\overline{I}\cap \overline{J} = \emptyset$ and let $Q\subset \R^{n-1}$ be a bounded, open set.
Consider the subsets
$\mathcal I= Q\times I$ and $\mathcal J=Q\times J$ in $\R^n$.
Let $h:\R^n\to \C$ be such that $h \in C^{0}(\R^n,\C)$,  $\im h|_{\mathcal J}=0$, $\re h|_{\mathcal J}$ is real analytic and $h$ has an analytic extension $\tilde  h$ in the $x_n$-variable into the upper complex half-plane.
Let $K\subset\R$ be a bounded interval containing $\conv(2I\cup 2J)$ and define
$\mathcal K=Q\times K$. Suppose that
 \begin{align*}
\|\tilde h\|_{H^{s_1}(\mathcal K\times[0,2])}+\| h\|_{H^{s_2}(\mathcal J)}\leq  \mathcal M
\end{align*}
for some $s_1>0$, $s_2>1$.
Then for any $s>0$ there exist $\nu>0$ and $C>0$ such that 
\begin{align*}
\|\tilde h\|_{L^2(\mathcal I\times[0,1])}\leq C\mathcal M\frac{1}{\left|\log \left( C\frac{\|  h\|_{H^{-s}(\mathcal J)}}{\mathcal M}\right) \right|^\nu}.
\end{align*}
\end{thm}

\begin{proof}
The proof follows along the lines of the proof of Theorem ~\ref{thm:analytic1D}, yielding analogous results after freezing the $x'$ variables in $Q$ and then integrating in these variables.\\

\textit{Step 1: Estimate on $\mathcal I\times [0,\tau]$.} Fixing the $x'$ variables, the estimate \eqref{eq:f+striptau} implies 
\begin{align*}
\|\tilde h(x',\cdot)\|_{L^2(I\times[0,\tau])}\leq C\tau^{\frac{1}{q}}\|\tilde h(x',\cdot)\|_{H^{s_1}(I\times[0,\tau])}.
\end{align*}
Integrating this with respect to the $x'$ variables in the domain $Q$ then results in
\begin{align*}
\|\tilde h\|_{L^2(\mathcal I\times[0,\tau])} \leq C\tau^{\frac{1}{q}}\|\tilde h\|_{L^2(Q,H^{s_1}(I\times[0,\tau]))}\leq C\tau^{\frac{1}{q}}\|\tilde h\|_{H^{s_1}(\mathcal I\times[0,\tau])}\leq C\tau^{\frac{1}{q}}\mathcal M.
\end{align*}
\medskip
\textit{Step 2: Estimate on $\mathcal I\times[\tau,1]$.} 
By Lemma ~\ref{lem:John}, after integrating and using H\"older's inequality and recalling that $\re \tilde{h}|_{J} = h|_{J}$, we obtain
\begin{align*}
\|\tilde h\|_{L^2(\mathcal J/4\times[0,1])}\leq C\| h\|_{H^1(\mathcal J)}^{\beta}\|\tilde h\|_{L^2(2\mathcal J\times[0,2])}^{1-\beta},
\end{align*}
where $k\mathcal J= Q\times kJ$.
Similarly, after applying $N\sim N_0-C\log \tau$ times the three balls inequality to $\re \tilde{h}$ and $\im \tilde{h}$, we infer
\begin{align*}
\|\tilde h\|_{L^2(\mathcal I\times[\tau,1])}\leq C \| h\|_{H^1(\mathcal J)}^{\beta\alpha^N} \|\tilde h\|_{L^2(\mathcal K\times[0,2])}^{1-\beta \alpha^N}.
\end{align*}
In order to introduce the desired spaces, we interpolate as in \eqref{eq:interpolation}, which results in
 \begin{align*}
\|\tilde h\|_{L^2(\mathcal I\times[\tau,1])}\leq C \left(\| h\|_{H^{-s}(\mathcal J)}^\theta\| h\|_{H^{s_2}(\mathcal J)}^{1-\theta}\right)^{\beta\alpha^N} \|\tilde h\|_{L^2(\mathcal K\times[0,2])}^{1-\beta \alpha^N}.
\end{align*}
Thus, estimating the last two terms by the constant $\mathcal M$ from the assumption, we infer
\begin{align*}
\|\tilde h\|_{L^2(\mathcal I\times[\tau,1])}\leq C \| h\|_{H^{-s}(\mathcal J)}^{\alpha_0\alpha^N} \mathcal M^{1-\alpha_0 \alpha^N}.
\end{align*}

\textit{Step 3: Optimization.} Combining the estimates from the previous steps and optimizing in $\tau>0$ as above, we obtain the desired result.
\end{proof}

\subsection{Analyticity and locality}
We recall the following Paley-Wiener type lemma from \cite{E81} (see also \cite{VE65}) which we will frequently rely on in the following sections:

\begin{thm}[Lemma 4.4. in \cite{E81}]
\label{thm:analytic_supp}
Let $p(\xi',\xi_n + i \tau)$ be continuous with respect to all variables $(\xi',\xi_n, \tau)$ for $\xi' \neq 0$ and $\tau \geq 0$ and analytic with respect to $\zeta_n = \xi_n + i \tau$ for $\tau>0$. Assume that the following growth condition holds
\begin{align*}
|p(\xi', \xi_n + i \tau )| \leq C (1+ |\xi'|+|\xi_n| + |\tau|)^{\alpha} 
\end{align*}
for some $\alpha \in \R$.
Then, $p(D', D_n): H^{s}(\R^n) \rightarrow H^{s-\alpha}(\R^n)$ is a bounded (pseudodifferential) operator which preserves the support on the negative (real) half-space, i.e. if $g\in C^{\infty}_c(\R^n_-)$, then also $\supp(pg)\subset \R^n_-$.
\end{thm}

\begin{proof}
The result follows from a deformation argument in the complex plane.
We first notice that by the Paley-Wiener-Schwartz theorem (see for instance Theorem 7.3.1 in \cite{HoermanderI}) for $g\in C_c^{\infty}({I})$ with ${I}\subset \R^n_-$ the function $\F g$ has an analytic extension into the complex plane $\C^n$ and for any $N\in \N$ there exists $C_N>1$ such that for $\tau \geq 0$, $\xi \in \R^n$
\begin{align}
\label{eq:PWS}
|\F g(\xi', \xi_n + i \tau)| \leq C_N(1+|\xi'|+|\xi_n|+|\tau|)^{-N}.
\end{align}
  
We next claim that
\begin{align}\label{eq:tau0}
\begin{split}
v(x', x_n ) 
&:= \int\limits_{\R^n} p(\xi', \xi_n) \F g(\xi', \xi_n) e^{i x \cdot \xi} d\xi\\
&= \int\limits_{\R}\int\limits_{\R^{n-1}} p(\xi', \xi_n + i \tau) \F g(\xi', \xi_n + i \tau) e^{i x\cdot \xi  - x_n \tau } d\xi' d\xi_n.
\end{split}
\end{align}
Indeed, using the Cauchy integral formula, we consider the paths $\R_- \cup (0,i\tau) \cup \big(\R_- \times \{i\tau\}\big)$ and the paths $\R_+ \cup (0,i\tau) \cup \big(\R_+ \times \{i\tau\}\big)$ for $\tau \geq 1$ (see Figure ~\ref{fig:analytic_supp}). Relying on the decay of $\F g$ stated in \eqref{eq:PWS}, we deduce the claimed identity where we have already omitted the semi-circles closing the paths in the Cauchy integrals. 
Passing to the limit $\tau \rightarrow \infty$ for $x_n> 0$ in \eqref{eq:tau0} then implies that $v(x)=0$. Thus, $\supp(v)\subset \R^n_-$.

\begin{figure}
\begin{center}
\begin{tikzpicture}[scale=1]

\pgfmathsetmacro{\tt}{1}
\pgfmathsetmacro{\ee}{.1}

\draw[->] (-5.2,0)--(5.2,0);
\draw[->] (0,-.2)--(0,1.7);
\node at (-.4,1.6) {$\im \xi_n$};
\node at (5.2,-.2) {$\re \xi_n$};
\draw[-] (-.05,\tt)--(.05,\tt);
\node at (-.15,\tt+.15) {$\tau$};

\draw[->, blue, line width=.7pt] (-5,0)--(-2.5-\ee/2,0);
\draw[-, blue, line width=.7pt] (-5,0)--(-\ee,0);
\draw[->, blue, line width=.7pt] (-\ee,0)--(-\ee,\tt/2);
\draw[-, blue, line width=.7pt] (-\ee,0)--(-\ee,\tt);
\draw[-, blue, line width=.7pt] (-5,\tt)--(-\ee,\tt);
\draw[<-, blue, line width=.7pt] (-2.5-\ee/2,\tt)--(-\ee,\tt);

\draw[-, blue, line width=.7pt] (5,0)--(\ee,0);
\draw[<-, blue, line width=.7pt] (2.5+\ee/2,0)--(\ee,0);
\draw[-, blue, line width=.7pt] (\ee,0)--(\ee,\tt);
\draw[<-, blue, line width=.7pt] (\ee,\tt/2)--(\ee,\tt);
\draw[->, blue, line width=.7pt] (5,\tt)--(2.5+\ee/2,\tt);
\draw[-, blue, line width=.7pt] (5,\tt)--(\ee,\tt);

\end{tikzpicture}
\end{center}
\caption{Paths for the Cauchy integrals in the proof of Theorem ~\ref{thm:analytic_supp}.
}
\label{fig:analytic_supp}
\end{figure}
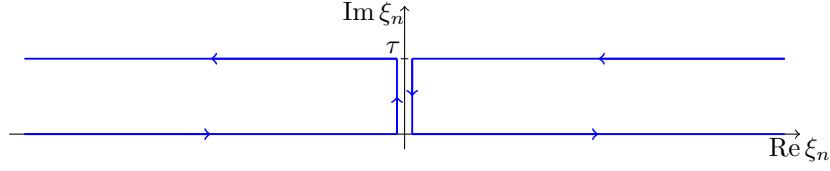

\end{proof}

\subsection{Further auxiliary results}

While Theorem ~\ref{thm:analytic1D} is the main tool in the sequel, we further collect a number of auxiliary results which we will frequently use in the next sections.

\begin{lem}
\label{lem:CR}
Let $h:\R\to\C$, $h\in C^0(\R,\C)\cap L^2(\R, \C)$,  have an analytic extension $\tilde h$ into the upper complex half-plane such that 
\begin{align}
\label{eq:apriori_bulk}
 \|\tilde h\|_{L^\infty(\C_+)} + \|\tilde{h} \|_{L^2(\C_+)} < \infty.
\end{align}
Let $s\in [\frac{1}{2},1)$. Then, for any interval $I\subset\R$ and constants $c>0$, $k> 1$, there exists a constant $C>0$ such that
\begin{align}\label{eq:bulkboundary}
\|\tilde{h}\|_{H^{\frac12}(I\times[0,c])}&\leq C\||D|^{-\frac{1}{2}}h\|_{H^{\frac{1}{2}}(\R)},
\\\label{eq:boundarybulk}
 \|h\|_{H^{-s}(I)}&< C \|\tilde{h}\|_{L^2(kI\times[0,c])}.
\end{align}
Here $|D|$ denotes the (tangential) Fourier multiplier corresponding to $|\xi|$.
\end{lem}

\begin{proof}
In order to prove \eqref{eq:bulkboundary}, we rely on the fact that the real and imaginary parts of $\tilde h$ are harmonic in the upper half-plane and on the following estimate for harmonic functions $u$ in $\R^2_+$ which satisfy the estimates in \eqref{eq:apriori_bulk}.
Indeed, for instance from the Fourier representation of the harmonic extension in the upper half-plane (which is a consequence of the $L^2$-assumption on $h$ in combination with the estimate \eqref{eq:apriori_bulk}), it follows that for $u$ harmonic with $|D|^{-\frac{1}{2}} u \in H^1(\R)$
\begin{align*}
\|u\|_{H^1(\R^2_+)} &\leq C \||D|^{-\frac{1}{2}} u\|_{H^1(\R)},\\
\|u\|_{L^2(\R^2_+)} &\leq C \||D|^{-\frac{1}{2}} u\|_{L^2(\R)}.
\end{align*}  
As a consequence, complex interpolation implies that
\begin{align*}
\|u\|_{H^{\frac{1}{2}}(\R^2_+)} \leq C \||D|^{-\frac{1}{2}} u \|_{H^{\frac{1}{2}}(\R)}.
\end{align*}
Thus, applying this to the real and imaginary parts, we infer the claim of \eqref{eq:bulkboundary}.

In order to infer \eqref{eq:boundarybulk}, we use the following observations (which are results of standard trace estimates): Let $(x,y)\in \R^2_+$ and let $\eta(x, y):= \eta_1(x)\eta_2(y)$ be a cut-off function with $\eta = 1$ on $I$ and $\supp(\eta) \subset k I \times [0,c]$. Then,
for $s\in[\frac{1}{2},1)$ and $f\in C^1(\R \times [0, c])$
\begin{align*}
\|f\|_{H^{-s}(I)}
&\leq  \|\eta f \|_{H^{-s}(\R)}\leq  \|\eta f \|_{H^{-\frac{1}{2}}(\R)}
\leq \|(1+(-\D_x))^{-\frac{1}{2}}(\eta f) \|_{H^{\frac{1}{2}}(\R)}\\
&\leq C(\|\nabla (1 + (-\D_x))^{-\frac{1}{2}} (\eta f) \|_{L^2(\R^2_+)} + \|(1+(-\D_x))^{-\frac{1}{2}}(\eta f)\|_{L^2(\R^{2}_+)})\\
&\leq C(\|\partial_x (1 + (-\D_x))^{-\frac{1}{2}} (\eta f) \|_{L^2(\R^2_+)} + \|\p_y (1 + (-\D_x))^{-\frac{1}{2}} (\eta f) \|_{L^2(\R^2_+)} + \|\eta f \|_{L^2(\R^{2}_+)})\\
&\leq C(\|\eta f\|_{L^2(\R^2_+)} + \|\p_y (1 + (-\D_x))^{-\frac{1}{2}} (\eta f) \|_{L^2(\R^2_+)} + \|\eta f \|_{L^2(\R^{2}_+)}).
\end{align*}
We apply this estimate both to $\re \tilde h$ and to $\im \tilde h$ assuming first that $\tilde{h}\in C^1(\R \times [0,c])$.
By virtue of the Cauchy-Riemann equations $\partial_x \tilde h=i\partial_y \tilde h$ (which strictly speaking only hold for $y>0$ but for which we argue first for $y= \epsilon$ and then pass to the limit $\epsilon \rightarrow 0$ in the end), we further estimate the term with $\partial_y$ as follows 
\begin{align*}
\|\p_y (1+(-\D_x))^{-\frac{1}{2}}&(\eta \re \tilde h)\|_{L^2(\R^2_+)}
\leq \| (1+(-\D_x))^{-\frac{1}{2}} \p_y (\eta \re \tilde h) \|_{L^2(\R^2_+)}\\
&\leq  \| (1+(-\D_x))^{-\frac{1}{2}} (\eta \p_y (\re \tilde h))  \|_{L^2(\R^2_+)}
+  \| (1+(-\D_x))^{-\frac{1}{2}} ( (\p_y \eta)\re \tilde h) \|_{L^2(\R^2_+)}\\
&\leq  \| (1+(-\D_x))^{-\frac{1}{2}} (\eta \p_x(\im \tilde h) )  \|_{L^2(\R^2_+)} +
\|  (\p_y \eta)\re \tilde h \|_{L^2(\R^2_+)}\\
&\leq  \| (1+(-\D_x))^{-\frac{1}{2}} \p_x(\eta  \im \tilde h)  \|_{L^2(\R^2_+)} +
\|(\p_x \eta) \im \tilde h \|_{L^2(\R^2_+)}+
\|(\p_y \eta) \re \tilde h \|_{L^2(\R^2_+)}
\\&\leq C \|\eta\|_{C^1(\R^2_+)}
(\| \re \tilde  h \|_{L^2(kI\times[0,c])}+\| \im \tilde h \|_{L^2(kI\times[0,c])}),
\end{align*} 
and similarly for $\im \tilde h$. Combining both estimates concludes the proof for $\tilde{h} \in C^1(\R \times [0,c])$. An approximation argument then implies the estimate under the claimed regularity assumption.
\end{proof}

We emphasize that the argument for \eqref{eq:boundarybulk} exploited that $s\in [\frac{1}{2},1)$. Nevertheless, this type of argument is not restricted to this regime: Accepting a higher loss in (relative) regularity, similar estimates can also be deduced for other values of $s\in \R$.

Relying on the same arguments as in the proof of Lemma ~\ref{lem:CR}, we obtain an analogue of \eqref{eq:boundarybulk} in higher dimensions.

\begin{cor}
\label{cor:nD_bdry_bulk}
Let $h:\R^n\to\C$, $h \in C^0(\R^n, \C)$, be such that in the $x_n$-variable  it has an analytic extension $\tilde h$ into the upper half-plane and let $s\in [\frac{1}{2},1)$. Then, for any $\mathcal I\subset \R^n$, $c>0$, $k > 1$ there exists a constant $C>0$ such that
\begin{align}
\label{eq:nd_boundry_bulk}
\|h\|_{H^{-s}(\mathcal{I})} \leq C \|\tilde{h}\|_{L^2(k \mathcal{I} \times [0,c])}.
\end{align}
\end{cor}

\begin{proof}
We argue as in the proof of Lemma ~\ref{lem:CR}. Letting $\eta(x,y)=\eta_1(x)\eta_2(y)$ with $x=(x', x_n)\in \R^n$, $y\in \R_+$ be a cut-off function with $\eta =1$ on $\mathcal{I}$ and $\supp(\eta)\subset k \mathcal{I}\times [0,c]$. Then for $s\in [\frac{1}{2},1)$ and $f\in C^1(\R^{n} \times [0,c])$
\begin{align*}
\|f\|_{H^{-s}(\mathcal{I})} 
&\leq \|\eta f\|_{H^{-s}(\R^n)}
\leq \|\eta f \|_{H^{- \frac{1}{2}}(\R^n)}
\leq \|\eta f\|_{L^2_{x'}(\R^{n-1},H^{-\frac{1}{2}}_{x_n}(\R))}\\
& \leq C\big(\|\eta f\|_{L^2(\R^{n-1}\times \R^2_+)}+\|\p_{x_n}(1+(-\D_{x_n}))^{-\frac{1}{2}}(\eta f)\|_{L^2(\R^{n-1}\times \R^{2}_+)} \\
&\qquad\quad+ \| \p_{y}(1+(-\D_{x_n}))^{-\frac{1}{2}}(\eta f)\|_{L^2(\R^{n-1}\times \R^{2}_+)} \big).
\end{align*}
Applying this estimate to $\re \tilde h$ and $\im \tilde h$ and continuing the proof as the one in Lemma ~\ref{lem:CR} by using the Cauchy-Riemann equations for $\tilde h$ in the $x_n$ and $y$ variables implies the desired result.
\end{proof}

For later use, we recall a version of the fractional Poincar\'e inequality:

\begin{lem}[Lemma 2.2 in \cite{CLR18}]
\label{lem:Poincare} 
Let $I \subset \R$ be an open interval, $f\in C^\infty_c(I)$ and $s\in (0,1)$. Then,
\begin{align}\label{eq:Poincare}
\|f\|_{L^2(I)}\leq C |I|^s \|f\|_{\dot H^s(I)}.
\end{align}
\end{lem}

Last but not least, we give a proof of the analytic pseudolocality of the operator $(-\D)^s g(x)$ for $g \in C_c^{\infty}(I)$ and $x \in \R \setminus \overline{I}$ (see also the discussion in Section 2 in \cite{L82}):

\begin{lem}
\label{lem:analytic_pseudo}
Let $g \in C_c^{\infty}(I)$ with $I\subset \R$ bounded. Further, let $s\in \R$. Then for $x \in J \subset \R \setminus \overline{I}$, the function $(-\D)^s g(x) = \F^{-1} \big(|\cdot|^{2s} \F g(\cdot)\big)(x)$ is (real) analytic.
\end{lem}

\begin{proof}
The argument follows along the same lines as the usual proof of pseudolocality of pseudodifferential operators. More precisely, for $\chi_N$ a cut-off function supported in $B_2$ and equal to one on $B_{1}$ with bounds as for instance in \cite[Lemma 1.1 in Chapter 5]{T80}, we consider the splitting
\begin{align*}
\F^{-1}\big( |\xi|^{2s} \F g(\xi)\big)(x)
= \F^{-1} \big(|\xi|^{2s} \chi_N(\xi) \F  g(\xi)\big)(x) + \F^{-1} \big(|\xi|^{2s} (1- \chi_N)(\xi) \F  g(\xi)\big)(x) .
\end{align*}
The function $\F^{-1} \big(|\xi|^{2s} \chi_N(\xi) \F  g(\xi)\big)(x)$ then by construction satisfies bounds of the type
\begin{align*}
|D_x^{\beta} \big(\F^{-1} \big(|\xi|^{2s} \chi_N(\xi) \F  g(\xi)\big)\big)(x)| \leq (C|\beta|)^{|\beta|},
\end{align*}
if $x\in J$ and $|\beta|\leq N$. 

For the second contribution we argue as usually for oscillatory integrals noting that 
\begin{align*}
\frac{d}{d\xi} e^{i (x-y)\xi} = {i(x-y)} e^{i(x-y)\xi},
\end{align*}
which leads to
\begin{align*}
&D_x^{\beta}\big(\F^{-1} \big(|\xi|^{2s} (1- \chi_N)(\xi) \F  g(\xi)\big)\big)(x)
 =  \int\limits_{\R} (i \xi)^{\beta} e^{ix\xi} |\xi|^{2s} (1-\chi_N)(\xi) \F g(\xi) d\xi \\
& \quad  = \int\limits_{\R} \int\limits_{\R}  e^{i(x-y)\xi}\left( -\frac{d}{i d\xi} \right)^{\beta+2+[|s|]}\big(|\xi|^{2s} (i\xi)^{\beta}(1-\chi_N)(\xi)\big) d\xi \frac{g(y)}{(x-y)^{\beta+ 2 + [|s|]}}dy,
\end{align*}
where $[\cdot]$ denotes the ceil function.
Due to the bounds on $\chi_N$ for $|\beta| \leq N-2-[|s|]$ and the fact that $\supp(\chi_N) \subset B_2$, we obtain that
\begin{align*}
\left|\left( \frac{d}{i d\xi} \right)^{\beta+2 + [|s|]}\big(|\xi|^{2s} (i\xi)^{\beta}(1-\chi_N(\xi))\big)\right| 
&\leq C |\beta + 2 + [|s|]|! \langle \xi \rangle^{-2}\\ &\leq (C |\beta|)^{|\beta|} \langle \xi \rangle^{-2} .
\end{align*}
As the index $N\geq 1$ can be chosen arbitrarily large in the cut-off function $\chi_N$ and since $|x-y|\geq \rho$ if $y \in \supp(g)$ and $x\in J$, this implies the desired analyticity bounds.
\end{proof}

\section{Quantitative Unique Continuation for Fractional Operators}
\label{sec:frac1D}

In this section we introduce the key ``comparison'' ideas in the model setting of the one-dimensional fractional Laplacian which we exploit in deducing the logarithmic stability estimates in the sequel. As already in the qualitative arguments in \cite{RS17b} the existence of a branch-cut for the complex extension of the symbol $|\xi|^s$ with $s\in \R \setminus \Z$ is the key mechanism underlying our arguments. We will present the details of this argument for the case $s= \frac{1}{2}$ and then for the case $s \in [\frac{1}{2},1)$. We however remark that slightly modified versions of this result hold for all values of $s\in \R \setminus \Z$. We emphasise that the logarithmic stability properties of the fractional Laplacian in any dimension had been previously addressed in \cite{RS17}, where quantitative unique continuation results had been deduced by means of robust Carleman inequalities (which in particular allow for the treatment of variable coefficients). The main novelty here is to provide a ``simple'' method which applies to constant coefficient operators.

\subsection{The half-Laplacian}
\label{sec:frac1/2}

We start by illustrating the ideas for the simplest case of  the half-Laplacian:

\begin{thm}
\label{thm:half}
Let $I, J \subset \R$ be  open, connected, bounded, non-empty intervals with $\overline{I}\cap \overline{J} = \emptyset$.
Then there exist constants $\mu>0$, $C>0$ such that for  $g\in C^\infty_c(I)$ we have
\begin{align*}
\|g\|_{H^1(I)}\leq Ce^{C\left(\frac{\|g\|_{H^1(I)}}{\|g\|_{L^2(I)}}\right)^\mu}\||D|g\|_{H^{-\frac{1}{2}}(J)}.
\end{align*}
\end{thm}

\begin{proof} 
Let us define the following functions
\begin{align*}
h_1(x)&=\left(|D|g+Dg\right)(x)=2\int_0^\infty e^{ix\xi}\xi \mathcal Fg(\xi)d\xi, \\
h_2(x)&=\left(|D|g-Dg\right)(x)=-2\int_{-\infty}^0 e^{ix\xi}\xi \mathcal Fg(\xi)d\xi.
\end{align*}
Hence, $\re h_j=|D|g$ and $\im h_j= (-1)^j g'$, where the dash denotes the one-dimensional differentiation with respect to $x\in \R$. In particular, by virtue of our assumptions, $\im h_j|_J=0$.
By analytic pseudolocality as in Lemma ~\ref{lem:analytic_pseudo}, for $x\in \R \setminus \overline{I}$ the function
\[|D|g(x)=\big(\mathcal F^{-1} \left(|\cdot|\mathcal Fg(\cdot)\right)\big)(x)
\] 
is analytic.
Moreover, the functions $h_j$, $j\in \{1,2\}$, have analytic extensions $\tilde h_j$ to the upper/lower  complex half-planes, respectively.
In order to apply Theorem ~\ref{thm:analytic1D} (and Remark ~\ref{rmk:lower}) to them, we first notice that for $j \in \{1,2\}$
\begin{align*}
\|\tilde h_j\|_{H^{\frac{1}{2}}(K\times[0,\pm2])}
&\overset{\eqref{eq:bulkboundary}}{\leq} C\||D|^{-\frac{1}{2}} h_j\|_{H^{\frac{1}{2}}(\R)}
\leq C\||D|^{\frac{1}{2}} g\|_{H^{\frac{1}{2}}(\R)}
\leq C\|g\|_{H^1(I)},\\
\| h_j\|_{H^{2}(J)}
& =\||D|g\|_{H^2(J)}\leq C\|g\|_{L^2(I)}.
\end{align*}
Here the second bound follows from the pseudolocality of the fractional Laplacian.

Next, we set $k=\frac{|I|+\dist(I,J)}{|I|}>1$. Therefore, by definition, $\overline{kI}\cap\overline J=\emptyset$. Thus, by  \eqref{eq:boundarybulk} and Theorem ~\ref{thm:analytic1D} for $s=\frac{1}{2}$, $s_1 = \frac{1}{2}$, $s_2 = 2$ and $M=C\|g\|_{H^1(I)}$, we have
\begin{align}
\label{eq:first_step}
\|h_j\|_{H^{-\frac{1}{2}}(I)}
\leq C \|g\|_{H^1(I)}\frac{1}{\left|\log\left( C\frac{\| h_j\|_{H^{-\frac{1}{2}}(J)}}{\|g\|_{H^1(I)}}\right) \right|^\mu}, \quad j=1,2,
\end{align}
where 
\begin{align*}
\| h_j\|_{H^{-\frac 1 2}(J)}&{=}\||D|g\|_{H^{-\frac 1 2}(J)}.
\end{align*}

Finally, we seek to bound the left hand side of \eqref{eq:first_step} from below by $\|g\|_{L^2(I)}$. This follows from the observation that $h_1 + h_2 = 2|D|g$ and the fact that
\begin{align*}
\||D|g\|_{H^{-\frac{1}{2}}(I)}=\sup_{\varphi\in\tilde H^{\frac{1}{2}}(I)} \frac{|(|D|g,\varphi)_I|}{\|\varphi\|_{H^{\frac{1}{2}}(I)}}\geq \frac{|(|D|g,g)_I|}{\|g\|_{H^\frac{1}{2}(I)}}\geq \frac{\|g\|_{\dot H^\frac{1}{2}(I)}^2}{\|g\|_{H^\frac{1}{2}(I)}}
\overset{\eqref{eq:Poincare}}{\geq} C \|g\|_{\dot H^\frac{1}{2}(I)},
\end{align*}
which in combination with \eqref{eq:first_step} implies
\begin{align*}
\|g\|_{L^2(I)}&\overset{\eqref{eq:Poincare}}{\leq}C\|g\|_{\dot H^\frac{1}{2}(I)}\leq C\||D| g\|_{H^{-\frac{1}{2}}(I)}
\leq C\left(\|h_1\|_{H^{-\frac{1}{2}}(I)}+ \|h_2\|_{H^{-\frac{1}{2}}(I)}\right)\\
&\leq C \|g\|_{H^1(I)}\frac{1}{\left|\log\left( C\frac{\||D|g\|_{H^{-\frac 1 2}(J)}}{\|g\|_{H^1(I)}}\right) \right|^\nu}.
\end{align*}
This is equivalent to the desired estimate from Theorem ~\ref{thm:half} with $\mu=\nu^{-1}$.
\end{proof}

Let us summarize our strategy for deducing the preceding stability estimates.
For $s= \frac{1}{2}$, the previous quantitative unique continuation result for the half-Laplacian used two ingredients:
\begin{itemize}
\item[(i)] As a first ingredient, we relied on the fact that the symbol $|\xi|$ is \emph{not analytic}. However, on the positive and negative lines, respectively, it is equal to the symbols $\xi$ or $-\xi$. The associated operators $D$ and $-D$ were used as ``comparison operators''. This allowed us to extend the
sum $(|D|+D)g$ and the difference $(|D|-D)g$ analytically into the upper/lower half-planes if $g$ was compactly supported (sufficient decay would have sufficed). Hence, quantitative propagation of smallness arguments could be used for analytic functions.
\item[(ii)] As our second ingredient, we exploited that the operator $D$ is a \emph{local operator}. As a consequence, the difference and sum operators $(|D|-D)g$ and $(|D|+D)g$ coincided with $|D|g$, if restricted to subsets outside of the support of $g$.
\end{itemize}
A similar strategy had been used in \cite{L82} to deduce qualitative antilocality results.

\begin{rmk}
\label{rmk:Hilbert}
We remark that with a similar argument as above, it is also possible to recover a variant of the stability estimate of Theorem ~\ref{thm:qucHilbert_a}: In order to avoid issues with the homogeneous function spaces, we consider $g= f'$ for some function $f \in C^{\infty}_c(I)$. Noting that the symbol of the Hilbert transform is given by $-i \sign(\xi)$ and defining our comparison functions to be 
\begin{align*}
h_1(x) = Hg(x) - i g(x), \  \ h_2(x) = Hg(x) + i g(x),
\end{align*}
carrying out the unique continuation arguments from Theorem ~\ref{thm:analytic1D} for the functions $h_j$ and returning from an estimate for $h_j$, $j\in \{1,2\}$, to an estimate for $f$, it is for instance possible to arrive at
\begin{align*}
\|f\|_{H^{\frac{1}{2}}(I)} \leq C \|f\|_{H^1(I)} \frac{1}{\left| \log\left( C \frac{\|H f\|_{L^2(J)}}{\|f\|_{H^1(I)}} \right) \right|^{\nu}}, \ f\in C_c^{\infty}(I).
\end{align*}
We emphasise that the choice of the norms is of course flexible (and could be shifted). 
\end{rmk}

While the construction of the comparison operator $D$ for $s=\frac{1}{2}$ (or $i \mbox{Id}$ for the Hilbert transform) was ``straight forward'', for more general operators this requires more care and information on both sides of the subset $I$.

\subsection{The fractional Laplacian for $s \in \left[\frac{1}{2},1\right)$}
\label{sec:sgen}
We now extend the logarithmic stability estimates to the whole class of (one-dimensional) operators of the type
\begin{align*}
|\xi|^{s}, \ s \in \left[\frac{1}{2},1 \right),
\end{align*}
and deduce associated observability type estimates/stability estimates for these. 

We pursue the same strategy as in the case $s=\frac{1}{2}$, however we now have to be more careful in the construction of the comparison operators, which will in general \emph{not} be local operators. For $s\in [\frac{1}{2},1)$ we provide the argument for the quantitative unique continuation result from Theorem ~\ref{thm:s_gen_1D}. Similar arguments (with an appropriate replacement of Lemma ~\ref{lem:CR}) would allow us to transfer this to general values of $s \in \R \setminus \Z$. In order to avoid case distinctions in the formulation of this, we refrain from stating the precise estimates in this more general setting and only discuss the situation of Theorem ~\ref{thm:s_gen_1D}.

\begin{proof}[Proof of Theorem ~\ref{thm:s_gen_1D}]
We mimic the comparison type argument from the previous section splitting the proof of the theorem into two steps.\\

\emph{Step 1: Set-up and comparison operators.}
We pursue the same strategy as in the previous section.
To this end, we introduce polar coordinates $\xi = |\xi| e^{i \varphi}$ with $\varphi$ in a suitable interval and consider the operators (defined in terms of their Fourier symbols)
\begin{align}\label{eq:Pjs}
P_{s}^1(\xi) = e^{2si(\varphi_+-\pi)}|\xi|^{2s},\ \varphi_+ \in \left[-\frac{1}{2}\pi,\frac{3}{2}\pi \right),\quad \ P_{s}^2(\xi) = e^{2si\varphi_-}|\xi|^{2s}, \ \varphi_- \in \left(-\frac{3}{2}\pi,\frac{1}{2}\pi \right],
\end{align} 
which are analytic extensions of $|\xi|^{2s}$ coinciding with $|\xi|^{2s}$  if $\xi \leq 0$ and $\xi \geq 0$, respectively. The branch-cuts are chosen to be on the negative/positive imaginary axes. Here the respective intervals for the angle refers to our choice of the branch of the logarithm. 
 By construction, $P_s^j(\xi)$ has an analytic extension into the upper/lower half-plane satisfying polynomial growth conditions. For $g\in C_c^{\infty}(\R^n)$ the Paley-Wiener theorem implies that $\hat{g}$ is an analytic function. Therefore,  
\begin{align*}
P_{s}^j(\xi) \hat{g}(\xi)
\end{align*}
has an analytic extension into the upper/lower half-plane. But then  $P_{s}^j(D)g(x)$ has support only on the half-line on the right/left of $I$ (see Theorem ~\ref{thm:analytic_supp}). Thus, we have that $P_{s}^j(D)g(x)|_{J_j}=0$.
\medskip

\emph{Step 2: Application of the quantitative analytic continuation estimates.}
Let us consider the functions
\begin{align*}
h_1(x)=\left(|D|^{2s}-P_{s}^1(D)\right)g(x), \quad h_2(x)=\left(|D|^{2s}-P_{s}^2(D)\right)g(x).
\end{align*}
Because of the previous discussion,
we have
\begin{align*}
h_j|_{J_j}=|D|^{2s} g|_{J_j}, \quad j=1,2,
\end{align*}
and therefore $\im h_j|_{J_j}=0$ and $\re h_j|_{J_j} = h_j|_{J_j}$ is real analytic by virtue of Lemma ~\ref{lem:analytic_pseudo}.
As the functions $h_j$ can be written as
\begin{align*}
h_1(x) &= \int\limits_{0}^{\infty} e^{i x \xi} |\xi|^{2s} (1- e^{-2s \pi i} ) \hat{g}(\xi) d\xi,\\
h_2(x) &= \int\limits_{-\infty}^{0} e^{i x \xi} |\xi|^{2s} (1- e^{-2s \pi i} ) \hat{g}(\xi) d\xi,
\end{align*}
they have analytic extensions into the upper/lower complex half-planes. As a consequence, the size of
\begin{align}
\label{eq:I}
\big(|D|^{2s}-P_{s}^j(D)\big)g(x)|_{I}, \quad j=1,2
\end{align}
can be quantified in terms of $|D|^{2s} g|_{J}$ by means of Theorem ~\ref{thm:analytic1D}. 
Heading towards an application of Theorem ~\ref{thm:analytic1D}, we note that
\begin{align}
\label{eq:E_estimates}
\begin{split}
\|\tilde h_j\|_{H^{\frac{1}{2}}(K\times[0,\pm2])}
&\overset{\eqref{eq:bulkboundary}}{\leq} C\||D|^{-\frac{1}{2}}h_j\|_{H^{\frac{1}{2}}(\R)}\leq C \||D|^{2s-\frac{1}{2}}g\|_{H^{\frac{1}{2}}(\R)}\leq C \|g\|_{H^{2s}(I)},\\
\| h_j\|_{H^{2}(J_j)}
&{=}\||D|^{2s}g\|_{H^2(J_1)}\leq C\|g\|_{L^2(I)}.
\end{split}
\end{align}
The last bound follows from the pseudolocality of the fractional Laplacian.
Thus, invoking \eqref{eq:boundarybulk} and Theorem ~\ref{thm:analytic1D} with $s_1 = \frac{1}{2}$, $s_2 = 2$ and $M=C\|g\|_{H^{2s}(I)}$, we infer that
\begin{align}
\label{eq:quant}
\|h_{j}\|_{H^{-s}(I)}\leq C \|g\|_{H^{2s}(I)}\frac{1}{\left|\log \left( C\frac{\| h_j\|_{H^{-s}(J_j)}}{\|g\|_{H^{2s}(I)}} \right) \right|^\nu}, \quad j=1,2.
\end{align}

It remains to deduce information on $g|_{I}$ from this, i.e., to bound the left hand side of \eqref{eq:quant} from below in terms of $g|_{I}$. To this end, we notice that 
\begin{align*}
h_1(x)+h_2(x)=P_{s}(D) g(x),
\end{align*} with $P_{s}(\xi)=2|\xi|^{2s}-P_{s}^1(\xi)-P_{s}^2(\xi)$, which on the real line turns into
\begin{align*}
P_{s}(\xi)=(1-e^{-2si\pi})|\xi|^{2s}.
\end{align*}
Hence,
\begin{align*}
\|P_{s}(D)g\|_{H^{-s}(I)}&=\sup_{v\in \tilde H^s(I)}\frac{|(P_{s}(D)g,v)_{L^2(I)}|}{\|v\|_{H^s(\R)}}
\geq \frac{|(P_{s}(D)g,g)_{L^2(I)}|}{\|g\|_{H^s(I)}}\\
&=\frac{|(P_{s}(D)g,g)_{L^2(\R)}|}{\|g\|_{H^s(I)}}
=\frac{|(P_{s}(\xi)\hat g,\hat g)_{L^2(\R)}|}{\|g\|_{H^s(I)}}\\
&\geq |1-e^{2is\pi}|\frac{\|g\|_{\dot H^s(I)}^2}{\|g\|_{H^s(I)}}
\overset{\eqref{eq:Poincare}}{\geq}  C |1-e^{-2is\pi}|\|g\|_{ H^s(I)}.
\end{align*}
As a consequence of this, we obtain the desired result
\begin{align}
\label{eq:quant1}
\begin{split}
\|g\|_{L^2(I )}
&\leq C\|P_s(D)g\|_{H^{-s}(I)}
\leq C(\|h_{1}\|_{H^{-s}(I)}+\|h_{2}\|_{H^{-s}(I)} )
\\&\leq C\|g\|_{H^{2s}(I)}\frac{1}{\left| \log \left( C\frac{\||D|^{2s}g\|_{H^{-s}(J)}}{\|g\|_{H^{2s}(I)}} \right) \right|^\nu}.
\end{split}
\end{align}
The claim of the theorem then follows from rearranging the estimate.
\end{proof}

In view of the application to quantitative Runge approximation which we have in mind (see Section ~\ref{sec:applic}), we make the following observation on the choice of the values $s_1, s_2$ and $M$ in the proof of Theorem ~\ref{thm:s_gen_1D}.

\begin{rmk}
\label{rmk:Runge_applic}
Instead of considering the norms in \eqref{eq:E_estimates}, it would also have been admissible (see Theorem ~\ref{thm:analytic1D} where this corresponds to the cases $s_1=\delta$ and $s_2=2$) to consider the bounds
\begin{align}
\label{eq:E_estimate_alt}
\begin{split}
\|\tilde h_j\|_{H^{\delta}(K \times [0,\pm2])} &\leq C \| g\|_{H^{2s + \delta-\frac{1}{2}}(I)},\\
\| h_j\|_{H^2(J)} &\leq C \|g\|_{L^2(I)}.
\end{split}
\end{align}
Here the first estimate follows from a boundary-bulk estimate for the harmonic extension.
Instead of arriving at the estimate \eqref{eq:quant1}, this choice of the parameters $s_1, s_2$ would have resulted in the bound
\begin{align}
\label{eq:quant2}
\|g\|_{H^s(I)} \leq C \|g\|_{H^{2s +\delta - \frac{1}{2}}(I)} \frac{1}{\left| \log\left( C \frac{\||D|^{2s} g\|_{H^{-s}(J)}}{\|g\|_{H^{2s+\delta-\frac{1}{2}}(I)}} \right) \right|^{\nu}}.
\end{align}
In the application of our estimates to quantitative Runge approximation in Section ~\ref{sec:applic} we will rely on estimates of this form.
\end{rmk}

\section{Extensions to More Complex Operators in Higher Dimensions}

We discuss new, possible extensions of our results to higher dimensions and more general operators. More precisely, we investigate combinations of local and nonlocal operators in different (possibly non-elliptic) combinations, possibly also involving pseudodifferential operators.  We only discuss the result and argument of Theorem ~\ref{thm:s_gen_nD} as one sample result in this spirit. Qualitative results in this spirit had earlier been discussed in \cite{DSV16, RS17a}. In \cite{RS17a} also quantitative results had been deduced for the fractional heat equation as a specific model operator by means of quantitative unique continuation results for the Caffarelli-Silvestre extension. In the present article we use our analytic comparison arguments to deduce new logarithmic estimates for the described, rather general class of nonlocal operators without relying the Caffarelli-Silvestre extension.

\begin{proof} [Proof of Theorem ~\ref{thm:s_gen_nD}]
We argue as in the proof of Theorem ~\ref{thm:s_gen_1D}, considering the functions
\begin{align*}
h_1(x)=\left(|D_{x_n}|^{2s}-P_{s}^1(D_{x_n})\right)g(x), \quad h_2(x)=  \left(|D_{x_n}|^{2s}-P_{s}^2(D_{x_n})\right)g(x),
\end{align*}
with  $P^j_s$ as in \eqref{eq:Pjs}.
By construction of $h_j$ and $P(D)=|D_{x_n}|^s+L(D)+m(D')$, we observe that
\begin{align*}
h_j|_{\mathcal J_j}=|D_{x_n}|^{2s}g|_{\mathcal J_j}=P(D) g|_{\mathcal J_j}.
\end{align*}
As in the previous sections, we have $\|\tilde h_j\|_{H^{\frac{1}{2}}(\mathcal{K}\times [0,\pm 2])} + \|h_j\|_{H^{2}(\mathcal{J})}\leq  C\|g\|_{H^{2s}(\mathcal I)} $. 
Therefore,  applying \eqref{eq:nd_boundry_bulk}  and Theorem ~\ref{thm:analyticnD} to $h_j$ with $s_1=\frac{1}{2}$, $s_2=2$, $\mathcal M=C\|g\|_{H^{2s}(\mathcal I)}$ and $I$ and $J$ substituted for slightly smaller subsets, we obtain
\begin{align*}
\|h_j\|_{H^{-s}(\mathcal I)}
\leq  C\|g\|_{H^{2s}(\mathcal I)}\frac{1}{\left| \log \left( C\frac{\|P(D)g\|_{H^{-s}(\mathcal J)}}{\|g\|_{H^{2s}(\mathcal I)}}\right) \right|^\nu}.
\end{align*}
Notice that, as in the proof of Theorem ~\ref{thm:s_gen_1D}, $h_1(x)+h_2(x)=P_s(D_{x_n})g(x)$, with $P_s(\xi_n)=2|\xi_n|^{2s}-P^1_s(\xi_n)-P^2_s(\xi_n)$, so
\begin{align}\label{eq:estimate1}
\|P_{s}(D_{x_n})g\|_{H^{-s}(\mathcal I)}
\leq  C \|g\|_{H^{2s}(\mathcal I)}\frac{1}{\left| \log \left( C\frac{\|P(D)g\|_{H^{-s}(\mathcal J)}}{\|g\|_{H^{2s}(\mathcal I)}}\right) \right|^\nu}.
\end{align}
In addition,
\begin{align*}
\|P_s(D_{x_n})g\|_{H^{-s}(\mathcal I)}
\geq \frac{|(P_s(D_{x_n})g,g)_{L^2(\mathcal I)}|}{\|g\|_{H^s(\mathcal I)}}
\geq |1-e^{2is\pi}|\frac{\|g\|_{L^2(Q, \dot H^s(I))}^2}{\|g\|_{H^s(\mathcal I)}}.
\end{align*}
Therefore,
\begin{align*}
\|g\|_{L^2(\mathcal I)}\leq C \|g\|_{L^2(Q, \dot H^s(I))}
\leq C \|P_s(D_{x_n})g\|_{H^{-s}(\mathcal I)}^{\frac 1 2}
\|g\|_{H^{s}(\mathcal I)}^{\frac 1 2},
\end{align*}
which together with \eqref{eq:estimate1} and the bound $\|g\|_{H^{s}(\mathcal{I})} \leq \|g\|_{H^{2s}(\mathcal{I})}$ yields
\begin{align}
\label{eq:estimate2}
\|g\|_{L^2(\mathcal I)}
\leq  C \|g\|_{H^{2s}(\mathcal I)}\frac{1}{\left| \log \left( C\frac{\|P(D)g\|_{H^{-s}(\mathcal J)}}{\|g\|_{H^{2s}(\mathcal I)}}\right) \right|^{\frac{\nu}{2}}}.
\end{align}
This leads to the desired result with $\mu=2\nu^{-1}$.
\end{proof}

\begin{rmk}
\label{rmk:Runge_applic_mult}
As in Remark ~\ref{rmk:Runge_applic} also in the proof of Theorem ~\ref{thm:s_gen_nD} we could  have chosen $\mathcal{M}\geq \|h_j\|_{H^{\delta}(\mathcal{K} \times [0,\pm 2])} + \|h_j\|_{H^{s}(\mathcal{J})}$ for some $\delta>0$. In this case, instead of \eqref{eq:estimate2}, we would have obtained the estimate
\begin{align}
\label{eq:estimate3}
\|g\|_{L^2(Q, H^{s}( I))}
\leq  C \|g\|_{H^{2s+\delta -\frac{1}{2}}(\mathcal I)}\frac{1}{\left| \log \left( C\frac{\|P(D)g\|_{H^{-s}(\mathcal J)}}{\|g\|_{H^{2s+\delta -\frac{1}{2}}(\mathcal I)}}\right) \right|^{\frac{\nu}{2}}}.
\end{align}
In our application to Runge approximation properties in the next section, we will rely on this variant of the estimate.
\end{rmk}

\begin{rmk}
We remark that as in \cite{RS17b} in this treatment of the multi-dimensional problem we have heavily used the fact that the quantitative unique continuation problem from above essentially reduces to a one-dimensional problem by slicing. It was thus important that either by locality or by the support assumption of $g$ we have that $L(D)g|_{\mathcal{J}_j} = m(D')g|_{\mathcal{J}_j}=0$.
\end{rmk}

\section{An Application: Quantitative Runge Approximation}
\label{sec:applic}

As an example of the applicability of our discussion from the previous sections, we prove quantitative Runge approximation properties for the operator 
\begin{align}
\label{eq:model}
\tilde{L}_{s} := \sum\limits_{j=1}^{n}(-\p_{x_j}^2)^s.
\end{align}

We remark that while \emph{qualitative} Runge approximation can be inferred from the works \cite{DSV16} and \cite{RS17a}, \emph{quantitative} Runge approximation results were not known previously for the model operator \eqref{eq:model}. While in principle it would be possible to obtain these results also from quantitative slicing arguments as in \cite{RS17a}, the estimates on the slices rely on relatively involved Carleman type estimates for the Caffarelli-Silvestre extension (for instance \cite{RS17} contains a new boundary-bulk Carleman estimate). Although these have the advantage of being very robust and in particular being able to deal with a rather large class of \emph{variable} coefficient operators, we show that for  many \emph{constant} coefficient operators, the easier quantitative analytic continuation arguments which were developed in this article suffice for proving quantitative Runge approximation results.  These could in turn be applied in order to derive stability estimates for the associated inverse problems, as for instance in \cite{RS17}.

As an example of this we prove the quantitative Runge approximation result of Theorem ~\ref{thm:quant_Runge}.

We recall that we assumed that zero was not a Dirichlet eigenvalue of the operator $\tilde{L}_s + q$.
This allows us to define the Poisson operator $P_q$ as the operator mapping $H^{s}(W)\ni f\mapsto u \in \widetilde{H}^s(\Omega)$, where $u$ is a solution associated with the operator $L_s = \tilde{L}_s + q$, and that this leads to a well-defined, bounded operator. Indeed, assuming that $q$ is such that zero is not a Dirichlet eigenvalue of the operator $L_s= \tilde{L}_s + q$, the (inhomogeneous) forward problem for the operator $L_s$, i.e. the solvability question for the problem
\begin{align*}
L_s u &= F \mbox{ in } \Omega,\\
u & = f \mbox{ in } \Omega_e,
\end{align*}
is well-posed for $f \in H^s(\Omega_e),\ F \in (\widetilde{H}^s(\Omega))^{\ast}$. As in the case of the fractional Laplacian, this directly follows from energy methods which also yield energy estimates of the form
\begin{align}
\label{eq:apriori}
\|u\|_{H^{s}(\Omega)} \leq C(\|F\|_{(H^{s}(\Omega))^{\ast}} + \|f\|_{H^{s}(\Omega_e)}).
\end{align}

Seeking to derive Theorem ~\ref{thm:quant_Runge}, we observe that as explained in Section 2 in \cite{RS17} it suffices to prove a quantitative unique continuation result for the dual equation
\begin{align}
\label{eq:dual}
\begin{split}
(\tilde L_s +q) w & = v \mbox{ in } \Omega,\\
w & = 0 \mbox{ in } \Omega_e.
\end{split}
\end{align}

Following the same arguments as in Section 2 in \cite{RS17}, Theorem ~\ref{thm:quant_Runge} can be deduced from the following quantitative unique continuation result:

\begin{prop}
\label{prop:dual_UCP}
Let $\Omega$, $W$, $q$ and $L_s$  be as in Theorem ~\ref{thm:quant_Runge}.
Let $w \in H^{s}_{\overline{\Omega}}$ be a solution to \eqref{eq:dual} with $v\in L^2(\Omega)$. Then, there exist $\mu>0$ and $C>0$ such that
\begin{align}
\label{eq:quant_UCP}
\|v\|_{H^{-s}(\Omega)}
\leq \frac{C}{\left| \log\left( C\frac{\|v\|_{L^2(\Omega)}}{\|{L}_s w\|_{H^{-s}(W)}} \right) \right|^{\mu}} \|v\|_{L^2(\Omega)}.
\end{align}
\end{prop}

\begin{proof}
We argue in two steps:

\emph{Step 1: Set-up.}
Without loss of generality, we may assume that $(B_4\setminus B_2) \subset W$. Indeed, as $(B_4 \setminus B_2) \cap [-1,1]^n = \emptyset$, we may else invoke analytic continuation in the form a three balls estimate (as in \eqref{eq:three_balls}, for instance) to extend $L_s w$ from $W$ to $B_4 \setminus B_2$ with Hölder bounds.

Next we consider $\ell\in \{1,\dots,n\}$ and define the functions
\begin{align*}
h_\ell^{j}(x) = \left( (-\p_{x_{\ell}}^2)^{s} - P_{s}^{j}(D_{x_\ell})\right)w(x), \quad j=1,2,
\end{align*}
with  $P^j_s$ as in \eqref{eq:Pjs}.
In the sequel, for simplicity, we only consider the case $\ell=n$ and define $\nabla'=(\p_{x_1},\dots, \p_{x_{n-1}})$. 
Let $W^n_1=(-1,1)^{n-1}\times (2,3)$ and $W^n_2=(-1,1)^{n-1}\times (-3,-2)$.
Noting that by the support condition for $w$ we have $\sum\limits_{j=1}^{n-1} (-\p_{x_j}^2)^s w(x',x_n) = 0$ for $(x',x_n)\in W^n:=W^n_1\cup W^n_2$, we apply the argument from Theorem ~\ref{thm:s_gen_nD} to  $h_n^{j}$ in the form of Remark ~\ref{rmk:Runge_applic_mult}. Thus, we obtain for $I:= (-1,1)$, $J_1 := (2,3)$, $J_2=(-3,-2)$, $Q=(-1,1)^{n-1}$  and $w$ playing the role of $g$, the estimate
\begin{align}
\label{eq:estimate4}
\|w\|_{L^2(Q\times H^{s}_{x_n}((-1,1)))}
\leq C M \frac{1}{\left| \log\left( \frac{\| L_s(D) w\|_{H^{-s}(W^n)}}{M} \right) \right|^{\frac{\nu}{2}}},
\end{align}
 where for $s\in [\frac{1}{2},1)$, $\delta>0$
\begin{align*}
M \geq C \|w\|_{H^{s+(\delta+s-\frac{1}{2})}(\Omega)} \geq C \|h_n^j\|_{H^{\delta}(\mathcal{K} \times [0,\pm 2])} .
\end{align*}
By symmetry we may also apply this for all other values of $\ell\in\{1,\dots,n\}$ which yields bounds as in \eqref{eq:estimate4} where the left hand side is replaced by $\|w\|_{L^2(Q\times H^{s}_{x_\ell}((-1,1)))}$ and in the right hand side the data are measured in terms of the expressions $\| L_s(D) w\|_{H^{-s}(W^\ell)}$ with $W^{\ell}:= (-1,1)^{\ell-1}\times J \times (-1,1)^{n-\ell}$. By the compact support condition for $w$, we obtain that
\begin{align*}
\sum\limits_{\ell=1}^{n} \|w\|_{L^2(Q \times H^{s}_{x_\ell}((-1,1)))} \geq \|w\|_{H^{s}((-1,1)^{n})} \ \mbox{ for } \ell\in\{1,\dots,n\}.
\end{align*}
Adding all the estimates of the type \eqref{eq:estimate4} and enlarging the bound from $\|L_s(D) w\|_{H^{-s}(W^\ell)}\leq \|L_s(D) w\|_{H^{-s}(W)}$, we thus obtain the bound 
\begin{align}
\label{eq:estimate5}
\|w\|_{H^{s}(\Omega)}
\leq C M  \frac{1}{\left| \log\left( \frac{\| L_s(D) w\|_{H^{-s}(W)}}{M} \right) \right|^{\frac{\nu}{2}}}.
\end{align}

\emph{Step 2: Rewriting the estimates in terms of $v$.}
In order to rewrite \eqref{eq:estimate5} in terms of $v$, we observe that by the a priori estimates for the operator $L_s$ (which follow from simple energy estimates as in the case of $(-\D)^s$, see for instance \cite{GSU16} and \eqref{eq:apriori}), we have
\begin{align*}
\|w\|_{H^{s}(\Omega)} &\leq C \|v\|_{H^{-s}(\Omega)}.
\end{align*}
Moreover, by the equation satisfied by $w, v$ and the structure of $\tilde{L}_s$ we have 
\begin{align}
\label{eq:apriori2}
\begin{split}
\|v\|_{H^{-s}(\Omega)}& \leq \|\tilde{L}_s w\|_{H^{-s}(\Omega)} + \|q w\|_{H^{-s}(\Omega)}
 \leq C \|w\|_{H^{s}(\Omega)} + \| q w\|_{L^2(\Omega)}\\
& \leq C (1+\|q\|_{L^{\infty}(\Omega)}) \|w\|_{H^{s}(\Omega)} .
\end{split}
\end{align}
Thus, there exists a constant $C>1$ (depending on $\|q\|_{L^{\infty}(\Omega)}$) such that
\begin{align*}
C^{-1}\|v\|_{H^{-s}(\Omega)} \leq \| w\|_{H^{s}(\Omega)} \leq C \|v\|_{H^{-s}(\Omega)}.
\end{align*}
Moreover, since the operator $\tilde{L}_s$ also satisfies the $\mu$-transmission condition, Vishik-Eskin \cite{VE65} estimates yield that for $s+\delta-\frac{1}{2}\in (-\frac{1}{2},\frac{1}{2})$ and $\delta \in (0,1)$ sufficiently small, 
\begin{align}
\|w\|_{H^{s+(s-\frac{1}{2}+\delta)}(\Omega)} \leq C \|v\|_{H^{-s+(s-\frac{1}{2}+\delta)}(\Omega)} \leq C \|v\|_{L^2(\Omega)} =:M.
\end{align}
Plugging this information into \eqref{eq:estimate5}, we infer
\begin{align}
\label{eq:estimate6}
\|v\|_{H^{-s}(\Omega)}
\leq C  \|v\|_{L^2(\Omega)}  \frac{1}{\left| \log\left( C \frac{\| L_s(D) w\|_{H^{-s}(W)}}{ \|v\|_{L^2(\Omega)}} \right) \right|^{\frac{\nu}{2}}},
\end{align}
which concludes the proof of \eqref{eq:quant_UCP}.
\end{proof}

Last but not least, we give a very rough sketch of how to conclude the result of Theorem ~\ref{thm:quant_Runge} from Proposition ~\ref{prop:dual_UCP}.

\begin{proof}[Proof of Theorem ~\ref{thm:quant_Runge}]
With Proposition ~\ref{prop:dual_UCP} in hand, the proof of Theorem ~\ref{thm:quant_Runge} essentially follows from the same duality arguments as the quantitative Runge approximation result in \cite{RS17}: In a first step, one shows that one may obtain associated orthonormal bases $\{\varphi_j\}\subset H^s_{\overline{W}}$ and $\{w_j\} \subset L^2(\Omega)$ such that for the operator $A:= i r_{\Omega}P_q: H^s_{\overline{W}} \rightarrow L^2(\Omega)$, where $i:H^s(\Omega) \rightarrow L^2(\Omega)$ denotes the inclusion map and $r_{\Omega}$ the restriction to $\Omega$ map, one has
\begin{align*}
A \varphi_j = \sigma_j w_j.
\end{align*}
This follows from the fact that $A$ is compact by Sobolev embedding, injective (by the boundedness of $W$ and the quantitative unique continuation results proved in the previous sections or alternatively by the results of \cite{RS17a} or of Theorem \ref{thm:Isakov}) with a dense range (which follows from the Hahn-Banach theorem and qualitative unique continuation results relying on analytic pseudolocality and the compact support of the test function). Then, with the singular value bases in hand, we may argue exactly as in the proof of Lemma 3.3 (a) in \cite{RS17} in order to obtain the quantitative Runge approximation result.
\end{proof}

\section*{Acknowledgements}
The authors would like to thank Jeremy Marzuola for drawing their attention towards the modified Hilbert transform.

\bibliographystyle{alpha}
\bibliography{citationsHT_new}

\end{document}